\documentclass[12pt,a4paper]{article}

\usepackage{amsmath}
\usepackage{amsfonts}
\usepackage{amssymb}
\usepackage{amsthm}

\theoremstyle{definition}

\newtheorem{lemma}{Lemma}[section]
\newtheorem{theorem}[lemma]{Theorem}
\newtheorem{claim}[lemma]{Claim}

\newtheorem{remark}[lemma]{Remark}
\newtheorem{definition}[lemma]{Definition}
\newtheorem{proposition}[lemma]{Proposition}

\renewenvironment{proof}{\textbf{Proof.}}{\qed}

\begin{document}

\title{Simultaneous Universal Pad\'{e} - Taylor Approximation}
\author{K. Makridis}
\maketitle

\begin{abstract}

We prove simultaneous Universal Approximation of a certain type of Pad\'{e} Approximants and of Taylor series with the same indexes. This is a generic phenomenon in $H(\Omega)$ for any simply connected domain $\Omega$, as well as in several other spaces. Our results are valid for one center of expansion and for several centers, as well.

\end{abstract}

AMS clasification numbers: 30K05, 30E10, 41A20.

\medskip


\section{Introduction} 

It is well known (\cite{ena}, \cite{dio}, \cite{tria}, \cite{tessera}, \cite{pente}, \cite{exi}) that for every simply connected domain $\Omega \subseteq \mathbb{C}$ and for every $\zeta \in \Omega$ there exists a holomorphic function $f \in H(\Omega)$ satisfying the following property: 

For every compact set $K \subseteq \mathbb{C} \setminus \Omega $ with connected complement and for every polynomial $P$ there exists a sequence $(\lambda_n)_{n \geq 1} \in \mathbb{N}$ such that:

\begin{itemize}

\item[(1)]
$\sup_{z \in K} |S_{\lambda_n}(f, \zeta)(z) - P(z)| \to 0$ as $n \to + \infty$

\item[(2)]
For every compact set $L \subseteq \Omega$ it holds:
$$ \sup_{z \in L} |S_{\lambda_n}(f, \zeta)(z) - f(z)| \to 0 $$ 

as $n \to + \infty$  

\end{itemize}

Here $S_{N}(f, \zeta)(z) = \sum_{k = 0}^{N} \frac{f^{(k)}(\zeta)}{k!} (z - \zeta)^k$ denote the partial sums of the Taylor expansion of $f$ with center $\zeta$. Furthermore, the set of all functions $f \in H(\Omega)$ satisfying the previous properties is dense and $G_{\delta}$ in $H(\Omega)$, where the space $H(\Omega)$ is endowed with the topology of uniform convergence on compacta. That is, universality of Taylor series is a generic property in $H(\Omega)$ for every simply connected domain $\Omega$.

Recently the partial sums $S_{N}(f, \zeta)(\cdot)$ have been replaced by some rational functions, namely the Pad\'{e} approximants of $f$ (\cite{epta}, \cite{okto}, \cite{ennia}, \cite{deka}, \cite{enteka}). There are two types of universal Pad\'{e} approximants. One of these types is when we fix a sequence $(p_n, q_n) \in \mathbb{N} \times \mathbb{N}$ with $p_n \to + \infty$ and in approximations $(1)$ and $(2)$ above we replace $S_{\lambda_n}(f, \zeta)(\cdot)$ by the Pad\'{e} approximants $[f; p_n / q_n]_{\zeta}(\cdot)$, which we assume that they exist and have a unique representation $[f; p / q]_{\zeta}(\cdot) = \frac{A(z)}{B(z)}$, where the functions $A(z)$ and $B(z)$ are polynomials $A(z) = \sum_{j = 0}^{p} c_{j} (z - \zeta)^{j}$ and $B(z) = \sum_{j = 0}^{q} b_{j} (z - \zeta)^{j}$ with $b_0 = 1$. This new universality is also generic in $H(\Omega)$. When we have two dense and $G_{\delta}$ sets in a complete metric space, their intersection is also a dense and $G_{\delta}$ set according to Baire's theorem. In this way we find a holomorphic function $f \in H(\Omega)$ which is a universal Taylor series and has universal Pad\'{e} approximants. But in the following approximations:

\begin{itemize}

\item[(1)]
$\sup_{z \in K} |S_{\lambda_n}(f, \zeta)(z) - P(z)| \to 0$, as $n \to + \infty$

\item[(2)]
$\sup_{z \in L} |S_{\lambda_n}(f, \zeta)(z) - f(z)| \to 0 $, as $n \to + \infty$ and that holds for every compact set $L \subseteq \Omega$

\item[(3)]
$\sup_{z \in K} |[f; p_{\mu_n} / q_{\mu_n}]_{\zeta}(z) - P(z)| \to 0$, as $n \to + \infty$

\item[(4)]
$\sup_{z \in L} |[f; p_{\mu_n} / q_{\mu_n}]_{\zeta}(z) - f(z)| \to 0$, as $n \to + \infty$ and that holds for every compact set $L \subseteq \Omega$

\end{itemize}

the indexes $\lambda_n$ and $p_{\mu_n}$ are not related. However,  repeating the proofs of genericities simultaneously we obtain that $\lambda_n = p_{\mu_n}$ (see also Corollary 2 of \cite{tessera}, \cite{dio}, \cite{tria}, \cite{dodeka}). This phenomenon is also generic in $H(\Omega)$.

We also obtain a variant of the above result valid simultaneously for several centers of expansion $\zeta \in \Omega$. This is the content of section $4$. In section $3$ we give a variant for formal series in the sense of Seleznev (\cite{dekatria}, \cite{tessera}). In section $5$ we prove a weaker result than the one in section $4$, again generic in $H(\Omega)$, for any simply connected domain $\Omega$, where the universal approximation is not required to be valid also on the boundary of $\Omega$. For Taylor series this kind of universality was obtain in the $70's$ by Luh and Chui - Parnes (\cite{dekatessera}, \cite{dekapente}, \cite{dekaexi}). The stronger notion of universality where the approximation is also valid on the boundary of $\Omega$ was obtain by V. Nestoridis in 1996 (\cite{ena}, \cite{dio}, \cite{tria}). If the universal approximation is valid on the boundary also, then the universal function $f$ has some very wild properties (\cite{ena}, \cite{tria}, \cite{dekaepta}). But if the universal approximation is not required to be valid on the boundary of $\Omega$, then the universal function can be smooth on the boundary. Thus, we obtain generic universality in $A(\Omega)$, provided that $\overline{\Omega}^{o} = \Omega$ and also $\{ \infty \} \cup \mathbb{C} \setminus \overline{\Omega}$ is connected (section $6$) and in a closed subspace of $A^{\infty}(\Omega)$ provided that $\{ \infty \} \cup \mathbb{C} \setminus \overline{\Omega}$ is connected (section $7$). 

Finally in section $8$ we prove a result where in a part of the boundary of $\Omega$ the universal approximation is valid while on another disjoint part of the boundary, the universal function is smooth. In section $2$ we include a few preliminaries mainly about Pad\'{e} approximants needed in the sequel.

\section{Preliminaries} 

\begin{definition} \label{definition 2.1}

Let $\zeta \in \mathbb{C}$, $(a_n)_{n \geq 0} \subseteq \mathbb{C}$ and $f$ a formal power series with center $\zeta$:
$$ f(z) = \sum_{k = 0}^{+ \infty} a_k (z - \zeta)^k $$

Now, for every $p, q \in \mathbb{N}$ we consider a function of the form:
$$ [f; p / q]_{\zeta}(z) = \frac{A(z)}{B(z)} $$

where the functions $A(z)$ and $B(z)$ are polynomials such that $deg(A(z)) \leq p$, $deg(B(z)) \leq q$, $B(\zeta) = 1$ and also the Taylor expansion of the function $[f; p / q]_{\zeta}(z) = \sum_{k = 0}^{+ \infty} b_k (z - \zeta)^k$ (with center $\zeta \in \mathbb{C}$) satisfies:
$$ a_k = b_k \; \text{for every} \; k = 0, \cdots p + q $$

If such a rational function exists, it is called the $(p, q)$ - Pad\'{e} approximant of $f$. Very often the power series $f$ is the Taylor development of a holomorphic function with center $\zeta$; a point in its domain of definition.

\end{definition}

\begin{remark} \label{remark 2.2}

According to Definition \ref{definition 2.1} we obtain that for $q = 0$ the $(p, 0)$ - Pad\'{e} approximant of $f$ exists trivially for every $p \in \mathbb{N}$, since:
$$ [f; p / 0]_{\zeta}(z) = \sum_{k = 0}^{p} a_k (z - \zeta)^k $$

for every $z \in \mathbb{C}$.
\end{remark}

\begin{remark} \label{remark 2.3}

For $q \geq 1$ Definition \ref{definition 2.1} does not necessarily implies the existence of Pad\'{e} approximants. However, if a Pad\'{e} approximant exists then it is unique as a rational funtion. It is known (\cite{dekaokto}) that a necessary and sufficient condition for the existence and uniqueness of the polynomials $A(z)$ and $B(z)$ above is that the following $q \times q$ Hankel determinant:

$$ D_{p, q}(f, \zeta) = det
\begin{vmatrix}
a_{p - q + 1} & a_{p - q + 2} & \cdots & a_{p} \\ 
a_{p - q + 2} & a_{p - q + 3} & \cdots & a_{p + 1} \\
\vdots & \vdots & \ddots &\vdots \\
a_{p} & a_{p + 1} & \cdots & a_{p + q + 1} \\ 
\end{vmatrix} $$

is not equal to $0$, i.e. $D_{p, q}(f, \zeta)  \neq 0$. In the previous determinant we set $a_k = 0$ for every $k < 0$. In addition, if $D_{p, q}(f, \zeta)  \neq 0$ we also write $f \in D_{p, q}(\zeta)$.

In this case, the $(p, q)$ - Pad\'{e} approximant of $f$ (with center $\zeta \in \mathbb{C}$) is given by the following formula:
$$ [f; p / q]_{\zeta}(z) = \frac{A(f, \zeta)(z)}{B(f, \zeta)(z)} $$

where: 

$$ A(f, \zeta)(z) = $$
$$ = det
\begin{vmatrix}
(z - \zeta)^{q} S_{p - q}(f, \zeta)(z) & (z - \zeta)^{q - 1} S_{p - q + 1}(f, \zeta)(z) & \cdots & S_{p}(f, \zeta)(z) \\ 
a_{p - q + 1} & a_{p - q + 2} & \cdots & a_{p + 1} \\
\vdots & \vdots & \ddots &\vdots \\
a_{p} & a_{p + 1} & \cdots & a_{p + q} \\ 
\end{vmatrix} $$

\bigskip

$$ B(f, \zeta)(z) = det
\begin{vmatrix}
(z - \zeta)^{q} & (z - \zeta)^{q - 1} & \cdots & 1 \\ 
a_{p - q + 1} & a_{p - q + 2} & \cdots & a_{p + 1} \\
\vdots & \vdots & \ddots &\vdots \\
a_{p} & a_{p + 1} & \cdots & a_{p + q} \\ 
\end{vmatrix} $$

with:

$$ S_k(f, \zeta)(z) = \begin{cases}
\sum_{i = 0}^{k} a_i (z - \zeta)^i, \; \text{if} \; k \geq 0 \\
0, \; \text{if} \; k < 0 
\end{cases} $$

The previous relations are called Jacobi formulas. Also, in this case, we notice that the polynomials $A(f, \zeta)(z)$ and $B(f, \zeta)(z)$ do not have any common zeros in $\mathbb{C}$, provided that $f \in D_{p, q}(\zeta)$.

\end{remark}

We will also make use of the following proposition.

\begin{proposition} \label{proposition 2.4} (\cite{dekaokto}) Let $f(z) = \frac{A(z)}{B(z)}$ be a rational function where the functions $A(z)$ and $B(z)$ are polynomials with $deg(A(z)) = p_0$ and $deg(B(z))$ $= q_0$. In addition, suppose that $A(z)$ and $B(z)$ do not have any common zero in $\mathbb{C}$. Then for every $\zeta \in \mathbb{C}$ such that $B(\zeta) \neq 0$ we have:

\begin{itemize}

\item[(1)]
$f \in D_{p_0, q_0} (\zeta)$

\item[(2)]
$f \in D_{p, q_0} (\zeta)$ for every $p \geq p_0$

\item[(3)]
$f \in D_{p_0, q} (\zeta)$ for every $q \geq q_0$

\end{itemize}

Moreover, for every $(p, q) \in \mathbb{N} \times \mathbb{N}$ with $p > p_0$ and $q > q_0$ we have:
$$ f \not\in D_{p, q} (\zeta) $$

In all cases above we obtain that $f(z) \equiv [f; p / q]_{\zeta}(z)$.

\end{proposition}

\section{Universality in the sense of Seleznev} 

In this section we prove Seleznev's type simultaneous universal Pad\'{e} - Taylor approximation. See \cite{dekatria}, \cite{tessera}, \cite{okto}.

Consider the space $\mathbb{C}^{\mathbb{N}}$ endowed with the Cartesian topology. A well - known result is that $\mathbb{C}^{\mathbb{N}}$ is a metrizable topological space; the same topology on $\mathbb{C}^{\mathbb{N}}$ can be induced by the following metric:
$$ \text{For every} \; a, b \in \mathbb{C}^{\mathbb{N}} \; \text{with} \; a \equiv (a_n)_{n \geq 0} \; \text{and} \; b \equiv (b_n)_{n \geq 0}  \; \text{we define}: $$
$$ \rho_c(a, b) = \sum_{n = 0}^{+ \infty} \frac{1}{2^n} \frac{|a_n - b_n|}{1 + |a_n - b_n|} $$

We know that $(\mathbb{C}^{\mathbb{N}}, \rho_c)$ is a complete metric space.

Another metric that can be introduced on $\mathbb{C}^{\mathbb{N}}$ giving a different topology from the Cartesian one is the following:
$$ \text{For every} \; a, b \in \mathbb{C}^{\mathbb{N}} \; \text{with} \; a \equiv (a_n)_{n \geq 0} \; \text{and} \; b \equiv (b_n)_{n \geq 0}  \; \text{we define}: $$
$$ \rho_d(a, b) = 
\begin{cases} 
2^{- n_0} &\mbox{if} \; a \neq b \; (\text{where} \; n_0 = min \{ n \in \mathbb{N} : a_n \neq b_n \}) \\ 
0 & \mbox{if} \; a = b 
\end{cases} $$

It is also true that $(\mathbb{C}^{\mathbb{N}}, \rho_d)$ is a complete metric space. Moreover, one can see that $\rho_c \leq 2 \rho_d$.

We present now the main theorem of this section.

\begin{theorem} \label{theorem 3.1} 

Let $F \subseteq \mathbb{N} \times \mathbb{N}$ be a non empty set containing exactly a sequence $(p_n, q_n)_{n \geq 0}$ such that $p_n \to + \infty$. Then there exists an element $a \equiv (a_n)_{n \geq 0} \in \mathbb{C}^{\mathbb{N}}$ such that the formal power series $f(z) = \sum_{n = 0}^{+ \infty} a_nz^n$ satisfies the following:

For every compact set $K \subseteq \mathbb{C} \setminus \{ 0 \}$ with connected complement and for every function $\psi \in A(K)$ there exists a subsequence $(p_{k_n}, q_{k_n})_{n \geq 0}$ of the sequence $(p_n, q_n)_{n \geq 0}$ such that:

\begin{itemize}

\item[(1)]
$f \in D_{p_{k_n}, q_{k_n}}(0)$ for every $n \in \mathbb{N}$

\item[(2)]
$\sup_{z \in K} |[f; p_{k_n} / q_{k_n}]_{0}(z) - \psi(z)| \to 0$ as $n \to + \infty$

\item[(3)]
$\sup_{z \in K} |S_{p_{k_n}}(f)(z) - \psi(z)| \to 0$ as $n \to + \infty$

\end{itemize}

Moreover, the set of all elements $a \equiv (a_n)_{n \geq 0} \in \mathbb{C}^{\mathbb{N}}$ satisfying $(1) - (5)$ is dense and $G_{\delta}$ in both spaces $(\mathbb{C}^{\mathbb{N}}, \rho_d)$ and $(\mathbb{C}^{\mathbb{N}}, \rho_c)$.

\end{theorem}

\begin{proof} Let $K \subseteq \mathbb{C} \setminus \{ 0 \}$ be a fixed compact set with connected complement and $\{ f_j \}_{j \geq 1}$ an enumeration of polynomials with coefficients in $\mathbb{Q} + i\mathbb{Q}$. Also, let $\psi \in A(K)$. We know from Mergelyan's theorem) that the polynomials $\{ f_j \}_{j \geq 1}$ is a dense subset of $A(K)$.

Now, for every $(p, q) \in F$ and for every $j, s \geq 1$ we consider the following sets:
$$ F(p, q, j, s) = \{ a \equiv (a_0, a_1, \cdots ) \in \mathbb{C}^{\mathbb{N}} : \; \text{the formal power series} $$
$$ f(z) = \sum_{n = 0}^{+ \infty} a_n z^n \; \text{satisfies} \; f \in D_{p, q}(0) \; \text{and} \; \sup_{z \in K} |[f; p / q]_{0}(z) - f_{j}(z)| < \frac{1}{s} \} $$

$$ E(p, j, s) = \{ a \equiv (a_0, a_1, \cdots ) \in \mathbb{C}^{\mathbb{N}} : \; \text{the formal power series} $$ 
$$ f(z) = \sum_{n = 0}^{+ \infty} a_n z^n \; \text{satisfies} \; \sup_{z \in K} |S_p(f)(z) - f_{j}(z)| < \frac{1}{s} \} $$

If $\mathcal{U}(K)$ is the set of all elements $a \equiv (a_n)_{n \geq 0} \in \mathbb{C}^{\mathbb{N}}$ satisfying $(1) - (3)$ for the specific compact set $K$, one can verify (by using Mergelyan's theorem) that:
$$ \mathcal{U}(K) = \bigcap_{j, s \geq 1} \bigcup_{(p, q) \in F} F(p, q, j, s) \cap E(p, j, s) $$

So, according to Baire's theorem, it is enough to prove the following:

\begin{claim} \label{claim 3.2} 

The sets $F(p, q, j, s)$ and $E(p, j, s)$ are open subsets of $(\mathbb{C}^{\mathbb{N}}, \rho_c)$ for every $(p, q) \in F$ and for every $j, s \geq 1$. This triavially implies that these sets are open in $(\mathbb{C}^{\mathbb{N}}, \rho_d)$ as well, since $\rho_c \leq 2 \rho_d$.

\end{claim}

\begin{claim} \label{claim 3.3} 

The set $\mathcal{U}(j, s) = \bigcup_{(p, q) \in F} F(p, q, j, s) \cap E(p, j, s)$ is a dense subset of $(\mathbb{C}^{\mathbb{N}}, \rho_d)$ for every $j, s \geq 1$. This trivially implies that $\mathcal{U}(j, s)$ is also dense in $(\mathbb{C}^{\mathbb{N}}, \rho_c)$.

\end{claim}

\subsection{The case of \boldmath{$F(p, q, j, s)$.}} 

We assume that $q \geq 1$ because for $q = 0$ the sets $F(p, q, j, s)$ and $E(p, j, s)$ coincide and the set $E(p, j, s)$ will be proven to be open.

Let $a \equiv (a_n)_{n \geq 0} \in \mathbb{C}^{\mathbb{N}}$ and $\varepsilon > 0$ be small enough. The number $\varepsilon > 0$ will be determined later on. Also, let $b \equiv (b_n)_{n \geq 0} \in \mathbb{C}^{\mathbb{N}}$ satisfying:
$$ \rho_{c}(a, b) < \varepsilon \;\;\; (3.1.1) $$

Consider now the formal power series $f(z) = \sum_{n = 0}^{+ \infty} a_nz^n$ and $g(z) = \sum_{n = 0}^{+ \infty} b_nz^n$. If $\varepsilon > 0$ is small enough, the first $p + q + 1$ coefficients of the formal power series of $f$ are close enough one by one to those of $g$. We know that $f \in D_{p, q}(0)$ (because $a \in F(p, q, j, s)$), so the Hankel determinant $D_{p, q}(f, 0)$ of $f$ is not equal to $0$. This determinant depends continuously on the first $p + q + 1$ coefficients of the formal power series of $f$ which are close enough one by one to those of $g$. Thus, the Hankel determinant $D_{p, q}(g, 0)$ of $g$ is also not equal to $0$ and so $g \in D_{p, q}(0)$.

We recall the Pad\'{e} approximants: 
$$ [f; p / q]_{0}(z) = \frac{A(f)(z)}{B(f)(z)} $$ 

and
$$ [g; p / q]_{0}(z) = \frac{A(g)(z)}{B(g)(z)} $$

\sloppy{where the polynomials $A(f)(z), B(f)(z), A(g)(z)$ and $B(g)(z)$ are given by the Jacobi formulas}. Since $a \in F(p, q, j, s)$ we have that $\sup_{z \in K} |[f; p / q]_{0}(z) - f_{j}(z)| < \frac{1}{s}$, while $f_j(z) \in \mathbb{C}$ for every $z \in K$; thus $[f; p / q]_{0}(z) \in \mathbb{C}$ for every $z \in K$, or equivalently, $B(f)(z) \neq 0$ for every $z \in K$ because the polynomials $A(f)(\cdot)$ and $B(f)(\cdot)$ do not have any common zeros in $\mathbb{C}$, provided that $f \in D_{p, q}(0)$.

By continuity, one obtains that there exists a $\delta >0$ such that:
$$|B(f)(z)| \geq \delta \; \text{for every} \; z \in K \;\;\; (3.1.2)$$ 
 
The polynomials $A(f)(z)$ and $B(f)(z)$ depend continuously on the first $p + q + 1$ coefficients of the formal power series of $f$, which are close enough one by one to those of $g$. Thus for $\varepsilon >0$ small enough, one obtains:
$$ |B(g)(z)| \geq \frac{\delta}{2} \; \text{for every} \; z \in K \;\;\; (3.1.3)$$ 

By the triangle inequality for the formal power series of $g$, we have:
$$ \sup_{z \in K} |[g; p / q]_{0}(z) - f_{j}(z)| \leq $$ 
$$ \sup_{z \in K} |[g; p / q]_{0}(z) - [f; p / q]_{0}(z)| + $$
$$ + \sup_{z \in K} |[f; p / q]_{0}(z) - f_{j}(z)| \;\;\; (3.1.4) $$

The term $\sup_{z \in K} |[f; p / q]_{0}(z) - f_{j}(z)|$ is strictly less than $\frac{1}{s}$, since $a \in F(p, q, j, s)$. So, it remains to control the term $\sup_{z \in K} |[g; p / q]_{0}(z) - [f; p / q]_{0}(z)|$.

By combining relations $(3.1.2)$ and $(3.1.3)$ we have that for every $z \in K$ it is:
$$ |[g; p / q]_{0}(z) - [f; p / q]_{0}(z)| = $$
$$ |\frac{A(f)(z)}{B(f)(z)} - \frac{A(g)(z)}{B(g)(z)}| = |\frac{A(f)(z)B(g)(z) - A(g)(z)B(f)(z)}{B(f)(z)B(g)(z)}| \leq $$
$$ \leq \frac{2}{(\delta)^2}|A(f)(z)B(g)(z) - A(g)(z)B(f)(z)| \leq $$
$$ \leq \frac{2}{(\delta)^2}|A(f)(z)||B(f)(z) - B(g)(z)| + $$
$$ + \frac{2}{(\delta)^2}|B(f)(z)||A(f)(z) - A(g)(z)| \;\;\; (3.1.5) $$

It is now clear that the right part of inequality $(3.1.5)$ can be arbitrary small, provided that $\varepsilon > 0$ is small enough; especially it can be strictly less than $\frac{1}{s} - \sup_{z \in K} |[f; p / q]_{0}(z) - f_{j}(z)|$. The result follows from the triangle inequality.

\subsection{The case of \boldmath{$E(p, j, s)$.}} 

Let $a \equiv (a_n)_{n \geq 0} \in \mathbb{C}^{\mathbb{N}}$ and $\varepsilon > 0$ be small enough. The number $\varepsilon > 0$ will be determined later on. Also, let $b \equiv (b_n)_{n \geq 0} \in \mathbb{C}^{\mathbb{N}}$ satisfying:
$$ \rho_{c}(a, b) < \varepsilon \;\;\; (3.2.1) $$

Consider now the formal power series $f(z) = \sum_{n = 0}^{+ \infty} a_nz^n$ and $g(z) = \sum_{n = 0}^{+ \infty} b_nz^n$. If $\varepsilon > 0$ is small enough, the first $p + 1$ coefficients of the formal power series of $f$ are close enough one by one to those of $g$.

By the triangle inequality, for the formal power series of $g$, we have:
$$ \sup_{z \in K} |S_p(g)(z) - f_j(z)| \leq $$
$$ \sup_{z \in K} |S_p(g)(z) - S_p(f)(z)| + $$
$$ + \sup_{z \in K} |S_p(f)(z) - f_j(z)| $$

The term $\sup_{z \in K} |S_p(f)(z) - f_j(z)|$ is strictly less than $\frac{1}{s}$,  since $a \in E(p, j, s)$. So, it remains to control the term $\sup_{z \in K} |S_p(g)(z) - S_p(f)(z)|$. 

It is obvious that the term $\sup_{z \in K} |S_p(g)(z) - S_p(f)(z)|$ can become arbitrary small (provided that $\varepsilon > 0$ is small enough) and so, it suffices to demand $\sup_{z \in K} |S_p(g)(z) - S_p(f)(z)| < \frac{1}{s} - \sup_{z \in K} |S_p(f)(z) - f_j(z)|$ and then the result follows from the triangle inequality. This completes the proof of Claim \ref{claim 3.2}.

\subsection{Density of \boldmath{$\mathcal{U}(j, s)$.}} 

In order to prove Claim \ref{claim 3.3}, we fix the parameters $j, s \geq 1$ and we want to prove that the set: 
$$ \mathcal{U}(j, s) = \bigcup_{(p, q) \in F} F(p, q, j, s) \cap E(p, j, s) $$ 

is dense in $(\mathbb{C}^{\mathbb{N}}, \rho_d)$.

Let $a \equiv (a_0, a_1, \cdots) \in \mathbb{C}^{\mathbb{N}}$ and $\varepsilon > 0$ be small enough. Let also $b \equiv (b_0, b_1, \cdots) \in \mathbb{C}^{\mathbb{N}}$ such that $\rho_d(a, b) < \varepsilon$. Consider the formal power series $f(z) = \sum_{n = 0}^{+ \infty} a_n z^n$ and $g(z) = \sum_{n = 0}^{+ \infty} b_n z^n$. We select an index $n_0 \in \mathbb{N}$ such that:
$$ \frac{1}{2^{n_0}} < \varepsilon $$

and we consider the polynomial $p(z) = b_{n_0} z^{n_0} + \cdots + b_1 z + b_0$. Since $f_j(z)$ is a polynomial and $0 \notin K$, the function:
$$ \frac{f_j(z) - p(z)}{z^{n_0 + 1}} $$

belongs to $A(K)$. By Mergelyan's theorem, since $K$ has connected complement, there exists a polynomial $t(z)$ such that:
$$ \sup_{z \in K} |\frac{f_j(z) - p(z)}{z^{n_0 + 1}} - t(z)| = \sup_{z \in K} |\frac{f_j(z) - p(z) - t(z)z^{n_0 + 1}}{z^{n_0 + 1}}| < \frac{1}{s \cdot \sup_{z \in K} |z^{n_0 + 1}|} $$

The previous relation implies that: 
$$ \sup_{z \in K} |f_j(z) - p(z) - t(z)z^{n_0 + 1}| < \frac{1}{s} $$

since for every $z \in K$ it is:
$$ |f_j(z) - p(z) - t(z)z^{n_0 + 1}| = \frac{|f_j(z) - p(z) - t(z)z^{n_0 + 1}||z^{n_0 + 1}|}{|z^{n_0 + 1}|} \leq $$
$$ \leq \frac{|f_j(z) - p(z) - t(z)z^{n_0 + 1}|}{|z^{n_0 + 1}|} \sup_{z \in K} |z^{n_0 + 1}| < \frac{1}{s} $$

Since $p_n \to + \infty$, there exists an element $(p_{k_{n_0}}, q_{k_{n_0}}) \in F$ satisfying:
$$ p_{k_{n_0}} > max(deg (t(z)z^{n_0 + 1}), n_0) $$

Consider now the following polynomial:
$$ h(z) = p(z) + t(z)z^{n_0 + 1} + dz^{p_{k_{n_0}}} = \sum_{n = 0}^{+ \infty} c_n z^n $$
$$ \text{where} \; c \equiv (c_0, c_1, \cdots) \in \mathbb{C}^{\mathbb{N}} $$ 

with $d \in \mathbb{C} \setminus \{ 0 \}$ and $|d|$ sufficiently small. It is clear that $deg(t(z)z^{n_0 + 1} + dz^{p_{k_{n_0}}}) = p_{k_{n_0}}$, so $deg(h(z)) = p_{k_{n_0}}$. Moreover, $h \in D_{p_{k_{n_0}}, q_{k_{n_0}}}(0)$ and also $h(z) = [h; p_{k_{n_0}} / q_{k_{n_0}}]_{0}(z)$. In addition:
$$ \sup_{z \in K} |f_j(z) - h(z)| = \sup_{z \in K} |f_j(z) - (p(z) + t(z)z^{n_0 + 1} + dz^{p_{k_{n_0}}})| = $$ 
$$ = \sup_{z \in K} |f_j(z) - p(z) - t(z)z^{n_0 + 1} - dz^{p_{k_{n_0}}}| \leq  $$
$$ \leq \sup_{z \in K} |f_j(z) - p(z) - t(z)z^{n_0 + 1}| + |d|\sup_{z \in K} |z^{p_{k_{n_0}}}| $$

We notice that the right part of the last inequality can become strictly less than $\frac{1}{s}$, provided that $|d| > 0$ is small enough.

By the definition of $h(z)$ we have that $c_n = b_n$ for every $n = 0, \cdots, n_0$. This implies that:
$$ \rho_d(b, c) < \frac{1}{2^{n_0}} < \varepsilon $$ 

By the triangle inequality, we obtain:
$$ \rho_d(a, c) \leq \rho_d(a, b) + \rho_d(b, c) < \varepsilon + \varepsilon = 2 \varepsilon $$

It remains to show that $c \in \mathcal{U}(j, s)$; this is almost obvious:

\begin{itemize}

\item[(1)]
$c \in F(p_{k_{n_0}}, q_{k_{n_0}}, j, s)$ since $h \in D_{p_{k_{n_0}}, q_{k_{n_0}}}(0)$ and the quantity: 
$$ \sup_{z \in K} |[h; p_{k_{n_0}} / q_{k_{n_0}}]_{0}(z) - f_j(z)| = \sup_{z \in K} |h(z) - f_j(z)| \leq $$
$$ \leq \sup_{z \in K} |f_j(z) - p(z) - t(z)z^{n_0 + 1}| + |d|\sup_{z \in K} |z^{p_{k_{n_0}}}| $$

can become strictly less than $\frac{1}{s}$, provided that $|d| > 0$ is small enough.

\item[(2)]
$ c \in E(p_{k_{n_0}}, j, s)$ since $S_{k_{n_0}}(h)(z) = h(z)$ for every $z \in \mathbb{C}$ and so the quantity: 
$$ \sup_{z \in K} |S_{k_{n_0}}(h)(z) - t_j(z)| $$

is exactly the same as in $(1)$.

\end{itemize}

So, according to Baire's theorem the set $\mathcal{U}(j, s)$ is a dense subset of $(\mathbb{C}^{\mathbb{N}}, \rho_d)$ and that holds for every $j, s \geq 1$.

In order to complete the proof we fix a sequence of compacts subsets $\{ K_n \}_{n \geq 1}$ of $\mathbb{C} \setminus \{ 0 \}$ with connected complements such that for every compact set $K \subseteq \mathbb{C} \setminus \{ 0 \}$ with connected complement, there exists an index $n \in \mathbb{N}$ satisfying $K \subseteq K_n$ (\cite{tessera}, \cite{dekatessera}). If $\mathcal{U}$ is the set of all formal power series satisfying the theorem, then we have:
$$ \mathcal{U} = \bigcap_{n = 1}^{+ \infty} \mathcal{U}(K_n) $$

Hence, by Baire's theorem we obtain that $\mathcal{U}$ is dense and $G_{\delta}$ subset of $(\mathbb{C}^{\mathbb{N}}, \rho_d)$.

\end{proof}

\section{Universality valid also on the boundary} 

The first paper where universal approximation was obtained to hold also on the boundary is \cite{ena}. In the present section we extend this for simultaneous Pad\'{e} - Taylor universal approximation. We recall the following well known lemmas.

\begin{lemma} \label{lemma 4.1}

(\cite{tria}) Let $\Omega$ be a domain in $\mathbb{C}$. Then there exists a sequence $\{ K_m \}_{m \geq 1}$ of compact subsets of $\mathbb{C} \setminus \Omega$ with connected complements, such that for every compact set $K \subseteq \mathbb{C} \setminus \Omega$ with connected complement, there exists an index $m \in \mathbb{N}$ satisfying $K \subseteq K_m$.

\end{lemma}

\begin{lemma} \label{lemma 4.2}

(Existence of exhausting family; \cite{dekaennia}) Let $\Omega$ be an open set in $\mathbb{C}$. Then there exists a sequence $\{ L_k \}_{k \geq 1}$ of compact subsets of $\Omega$ such that:

\begin{itemize}

\item[(1)] 
$L_k \subseteq L_{k + 1}^o$ for every $k \in \mathbb{N}$.

\item[(2)]
For every compact $L \subseteq \Omega$ there exists a $k \in \mathbb{N}$ such that $L \subseteq L_k$.

\item[(3)]
Every connected component of $\tilde{\mathbb{C}} \setminus L_k$ contains at least one connected component of $\tilde{\mathbb{C}} \setminus \Omega$.
\end{itemize}
 
\end{lemma}

We present now the main theorem of this section.

\begin{theorem} \label{theorem 4.3}

Let $\Omega \subseteq \mathbb{C}$ a simply connected domain and $L \subseteq \Omega$ be a compact set. Also, let $F \subseteq \mathbb{N} \times \mathbb{N}$ be a non empty set containing exactly a sequence $(p_n, q_n)_{n \geq 1}$ such that $p_n \to + \infty$. Then, there exists a function $f \in H(\Omega)$ satisfying the following:
\\
\indent
For every compact set $K \subseteq \mathbb{C} \setminus \Omega$ with connected complement and for every function $h \in A(K)$, there exists a subsequence $(p_{k_n}, q_{k_n})_{n \geq 1}$ of the sequence $(p_n, q_n)_{n \geq 1}$ such that:
\begin{itemize}

\item[(1)]
$f \in D_{p_{k_n}, q_{k_n}}(\zeta)$ for every $n \in \mathbb{N}$ and for every $\zeta \in L$

\item[(2)]
$\sup_{\zeta \in L} \sup_{z \in K} |S_{p_{k_n}}(f, \zeta)(z) - h(z)| \to 0$ as $n \to + \infty$

\item[(3)]
$\sup_{\zeta \in L} \sup_{z \in K} |[f; p_{k_n} / q_{k_n}]_{\zeta}(z) - h(z)| \to 0$ as $n \to + \infty$

\item[(4)]
For every compact set $J \subseteq \Omega$ holds:
$$ \sup_{\zeta \in L} \sup_{z \in J} |S_{p_{k_n}}(f, \zeta)(z) - f(z)| \to 0 \; \text{as} \; n \to + \infty $$

\item[(5)]
For every compact set $J \subseteq \Omega$ it holds:
$$ \sup_{\zeta \in L} \sup_{z \in J} |[f; p_{k_n} / q_{k_n}]_{\zeta}(z) - f(z)| \to 0 \; \text{as} \; n \to + \infty $$

\end{itemize}  

Moreover, the set of all functions satisfying $(1) - (5)$ is dense and $G_\delta$ in $H(\Omega)$.

\end{theorem}

\begin{proof} Let us first consider an enumeration $\{ f_j \}_{j \geq 1}$ of polynomials with coefficients in $\mathbb{Q} + i\mathbb{Q}$.

Now, for every $(p, q) \in F$ and for every $j, s, m, k \geq 1$ we consider the following sets:
$$ A(K_m, p, j, s) = \{ f \in H(\Omega) : \sup_{\zeta \in L} \sup_{z \in K_m} |S_p(f,\zeta)(z) - f_j(z)| < \frac{1}{s}  \} $$
$$ B(K_m, p, q, j, s) = \{ f \in H(\Omega) : \; \text{for every} \; \zeta \in L \; \text{it holds} \; f \in D_{p,q}(\zeta) $$
$$ \text{and} \; \sup_{\zeta \in L} \sup_{z \in K_m} |[f; p / q]_{\zeta}(z) - f_j(z)| < \frac{1}{s} \} $$
$$ C(L_k, p, s) = \{ f \in H(\Omega) : \sup_{\zeta \in L} \sup_{z \in L_k} |S_p(f, \zeta)(z) - f(z)| < \frac{1}{s} \} $$
$$ D(L_k, p, q, s) = \{ f \in H(\Omega) : \; \text{for every} \; \zeta \in L \; \text{it holds} \; f \in D_{p,q}(\zeta) $$
$$ \text{and} \; \sup_{\zeta \in L} \sup_{z \in L_k} |[f; p / q]_{\zeta}(z) - f(z)| < \frac{1}{s} \} $$

One can verify (by using Mergelyan's theorem) that the set of all functions satisfying $(1) - (5)$ is precisely the following class:
$$ \mathcal{X} = \bigcap_{j, m, k, s \geq 1} 
\bigcup_{(p, q) \in F} A(K_m, p, j, s) \cap B(K_m, p, q, j, s) \cap C(L_k, p, s) \cap  D(L_k, p, q, s) $$

So, according to Baire's theorem it is enough to prove the following:

\begin{claim} \label{claim 4.4}
\sloppy{The sets $A(K_m, p, j, s)$, $B(K_m, p, q, j, s)$, $C(L_k, p, s)$ and} $D(L_k, p, q, s)$ are open subsets of $H(\Omega)$ for every $(p, q) \in F$ and for every $m, k, j, s \geq 1$.
\end{claim}

\begin{claim} \label{claim 4.5}
The set $\mathcal{X}(m, k, j, s) = \bigcup_{(p, q) \in F} A(K_m, p, j, s) \cap B(K_m, p, q, j, s) \cap C(L_k, p, s) \cap  D(L_k, p, q, s)$ is a dense subset of $H(\Omega)$ for every $m, k, j, s \geq 1$.
\end{claim}

Since $H(\Omega)$ endowed with the usual topology is a complete metric space, Baire's theorem implies that $\mathcal{X}$ is a dense and $G_{\delta}$ subset of $H(\Omega)$.

For $q = 0$ we automatically have that for every function $f \in H(\Omega)$ it holds $f \in D_{p, 0}(\zeta)$ for every $\zeta \in L$ and for every $p \geq 0$; in that case $[f; p / 0]_{\zeta} = S_p(f, \zeta)$. So, we restrict our attention to the case $q \geq 1$.

\subsection{The case of \boldmath{$A(K_m, p, j, s)$.}} 

We fix the parameters $(p, q) \in F$ and let $m, j, s \geq 1$. Let also $f \in A(K_m, p, j, s)$ and $\rho$ denote the usual metric for $H(\Omega)$. We want to select an $\varepsilon > 0$ such that for every $g \in H(\Omega)$ with $\rho(f, g) < \varepsilon$ it follows that $g \in A(K_m, p, j, s)$. The number $\varepsilon > 0$ will be determined later on.

Since $L$ is a compact set we have that $d \equiv d(L, \mathbb{C} \setminus \Omega) > 0$. Thus, for every $\zeta \in L$ it holds $B(\zeta, \frac{d}{2}) \subseteq \overline{B}(\zeta, \frac{d}{2}) \subseteq \Omega$. By compactness we can find an index $N \in \mathbb{N}$ and $\zeta_1, \cdots, \zeta_N \in L$ such that:
$$ L \subseteq \bigcup_{n = 1}^{N} B(\zeta_n, \frac{d}{2}) \subseteq \bigcup_{n = 1}^{N} \overline{B}(\zeta_n, \frac{d}{2}) \subseteq \Omega $$

One can see that if $\varepsilon > 0$ is small enough, there exists an index $n_0 \in \mathbb{N}$ with $n_0 > p + 1$ such that $\bigcup_{n = 1}^{N} \overline{B}(\zeta_n, \frac{d}{2}) \subseteq L_{n_0}$ while at the same time the quantity $\sup_{\zeta\in L_{n_0}} |f(\zeta) - g(\zeta)|$ is also small (see also Lemma \ref{lemma 4.2}).

By the Cauchy estimates, the first $p + 1$ Taylor coefficients of $g$ with center $\zeta \in L$ are uniformly close one by one to the corresponding Taylor coefficients of $f$. By the triangle inequality, we have:
$$ \sup_{\zeta \in L} \sup_{z \in K_m} |S_p(g, \zeta)(z) - f_j(z)| \leq $$
$$ \leq \sup_{\zeta \in L} \sup_{z \in K_m} |S_p(g, \zeta)(z) - S_p(f, \zeta)(z)| + \sup_{\zeta \in L} \sup_{z \in K_m} |S_p(f, \zeta)(z) - f_j(z)| \;\;\; (4.1.1)$$

Since $\frac{1}{s} - \sup_{\zeta \in L} \sup_{z \in K_m} |S_p(f, \zeta)(z) - f_j(z)| > 0$, we can obtain for $\varepsilon > 0$ small enough the following relation:
$$ \sup_{\zeta \in L} \sup_{z \in K_m} |S_p(g, \zeta)(z) - S_p(f, \zeta)(z)| < \frac{1}{s} - \sup_{\zeta \in L} \sup_{z \in K_m} |S_p(f, \zeta)(z) - f_j(z)| \;\;\; (4.1.2) $$ 

By combining relations $(4.1.1)$ and $(4.1.2)$ we aquire: 
$$ \sup_{\zeta \in L} \sup_{z \in K_m} |S_p(g, \zeta)(z) - f_j(z)| < \frac{1}{s} $$

It follows that $g \in A(K_m, p, j, s)$ and so $A(K_m, p, j, s)$ is open in $H(\Omega)$.

\subsection{The case of \boldmath{$B(K_m, p, q, j, s)$.}} 

We fix the parameters $(p, q) \in F$ and let $j, s \geq 1$. Let also $f \in B(K_m, p, q, j, s)$. We want to select an $\varepsilon > 0$ such that for every $g \in H(\Omega)$ with $\rho(f, g) < \varepsilon$ it follows that $g \in B(K_m, p, q, j, s)$. The number $\varepsilon > 0$ will be determined later on.

Since $f \in B(K_m, p, q, j, s)$ we have that $f \in D_{p, q}(\zeta)$ for every $\zeta \in L$. Moreover, the Hankel determinant $D_{p, q}(f)(\zeta)$ for $f$ is not equal to $0$ and that holds for every $\zeta \in L$. By continuity, there exists a $\delta > 0$ such that:
$$ |D_{p, q}(f)(\zeta)| > \delta \; \text{for every} \; \zeta \in L $$

Since $\rho(f, g) < \varepsilon$, we suppose that the first $p + q + 1$ Taylor coefficients of $g$ are uniformly close enough one by one to the corresponding Taylor coefficients of $f$, provided that $\varepsilon > 0$ is small enough. Again by continuity, we obtain:
$$ |D_{p, q}(g)(\zeta)| > \frac{\delta}{2} \; \text{for every} \; \zeta \in L $$

The last relation implies that $g \in D_{p, q}(\zeta)$ for every $\zeta \in L$. By the triangle inequality, we have:
$$ \sup_{\zeta \in L} \sup_{z \in K_m} |[g; p / q]_{\zeta}(z) - f_j(z)| \leq $$
$$ \leq \sup_{\zeta \in L} \sup_{z \in K_m} |[g; p / q]_{\zeta}(z) - [f; p / q]_{\zeta}(z)| + \sup_{\zeta \in L} \sup_{z \in K_m} |[f; p / q]_{\zeta}(z) - f_j(z)| \;\;\; (4.2.1) $$

We recall the Pad\'{e} approximants:
$$ [f; p / q]_{\zeta}(z) = \frac{A(f, \zeta)(z)}{B(f, \zeta)(z)} $$

and
$$ [g; p / q]_{\zeta}(z) = \frac{A(g, \zeta)(z)}{B(g, \zeta)(z)} $$

where the polynomials $A(f, \zeta)(z)$, $B(f, \zeta)(z)$, $A(g, \zeta)(z)$ and $B(g, \zeta)(z)$ are given by the Jacobi formulas and their coefficients vary continuously on $\zeta \in L$. Since $\sup_{\zeta \in L} \sup_{z \in K_m} |[f; p / q]_{\zeta}(z) - f_j(z)| < \frac{1}{s}$, we have that the function $[f; p / q]_{\zeta}(z)$ takes only finite values for every $z \in K_m$ and for every $\zeta \in L$. Thus, $B(f, \zeta)(z) \neq 0$ for every $\zeta \in L$ and for every $z \in K_m$. By continuity, one obtains that there exists a $\delta' >0$ such that:
$$ |B(f, \zeta)(z)| \geq \delta' \; \text{for every} \; \zeta \in L \; \text{and for every} \; z \in K_m $$ 
 
Since the first $p + q + 1$ Taylor coefficients of $f$ are uniformly close enough one by one to those of $g$ (provided that $\varepsilon > 0$ is small enough), one obtains:
$$ |B(g, \zeta)(z)| \geq \frac{\delta'}{2} \; \text{for every} \; \zeta \in L \; \text{and for every} \; z \in K_m $$ 

By the triangle inequality it holds:
$$ |[g; p / q]_{\zeta}(z) - [f; p / q]_{\zeta}(z)| = |\frac{A(f, \zeta)(z)}{B(f, \zeta)(z)} - \frac{A(g, \zeta)(z)}{B(g, \zeta)(z)}| = $$
$$ = |\frac{A(f, \zeta)(z)B(g, \zeta)(z) - A(g, \zeta)(z)B(f, \zeta)(z)}{B(f, \zeta)(z)B(g, \zeta)(z)}| \leq $$
$$ \leq \frac{2}{(\delta')^2}|A(f, \zeta)(z)B(g, \zeta)(z) - A(g, \zeta)(z)B(f, \zeta)(z)| \leq $$
$$ \leq \frac{2}{(\delta')^2}|A(f, \zeta)(z)||B(f, \zeta)(z) - B(g, \zeta)(z)| + $$
$$ + \frac{2}{(\delta')^2}|B(f, \zeta)(z)||A(f, \zeta)(z) - A(g, \zeta)(z)| $$

Thus, provided that $\varepsilon > 0$ is small enough, we obtain: 
$$ \sup_{\zeta \in L} \sup_{z \in K_m} |[g; p / q]_{\zeta}(z) - [f; p / q]_{\zeta}(z)| < $$
$$ < \frac{1}{s} - \sup_{\zeta \in L} \sup_{z \in K_m} |[f; p / q]_{\zeta}(z) - f_j(z)| \;\;\; (4.2.2) $$

By combining relations $(4.2.1)$ and $(4.2.2)$ we acquire: 
$$ \sup_{\zeta \in L} \sup_{z \in K_m} |[g; p / q]_{\zeta}(z) - f_j(z)| < \frac{1}{s} $$

It follows that $g \in B(K_m, p, q, j, s)$ and so $B(K_m, p, q, j, s)$ is open in $H(\Omega)$.

\subsection{The case of \boldmath{$C(L_k, p, s)$.}} 

We fix the parameters $(p, q) \in F$ and let $k, s \geq 1$. Let also $f \in C(L_k, p, s)$. We want to select an $\varepsilon > 0$ such that for every $g \in H(\Omega)$ with $\rho(f, g) < \varepsilon$ it follows that $g \in C(L_k, p, s)$. The number $\varepsilon > 0$ will be determined later on. The proof is similar to the one of subsection $4.1$, except from the following difference: 
$$ \sup_{\zeta \in L} \sup_{z \in L_k} |S_p(g, \zeta)(z) - g(z)| \leq \sup_{\zeta \in L} \sup_{z \in L_k} |S_p(g, \zeta)(z) - S_p(f, \zeta)(z)| + $$
$$ +\sup_{\zeta \in L} \sup_{z \in L_k} |S_p(f, \zeta)(z) - f(z)| + \sup_{z \in L_k} |f(z) - g(z)| \;\;\; (4.3.1) $$ 

If $\varepsilon > 0$ is small, one obtains that the quantity $\sup_{z \in L_k} |f(z) - g(z)|$ is also small. By the Cauchy estimates, we can achieve the following: 
$$ \sup_{\zeta \in L} \sup_{z \in L_k} |S_p(g, \zeta)(z) - S_p(f, \zeta)(z)| < \frac{1}{s} - \sup_{\zeta \in L} \sup_{z \in L_k} |S_p(f, \zeta)(z) - f(z)| - $$
$$ - \sup_{z \in L_k} |f(z) - g(z)| \;\;\; (4.3.2) $$

for $\varepsilon > 0$ small. By combining relations $(4.3.1)$ and $(4.3.2)$ we obtain: 
$$ \sup_{\zeta \in L} \sup_{z \in L_k} |S_p(g, \zeta)(z) - g(z)| < \frac{1}{s}. $$

It follows that $g \in C(L_k, p, s)$ and so $C(L_k, p, s)$ is open in $H(\Omega)$.

\subsection{The case of \boldmath{$D(L_k, p, q, s)$.}} 

We fix the parameters $(p, q) \in F$ and let $k, s \geq 1$. Let also $f \in D(L_k, p, q, s)$. We want to select an $\varepsilon > 0$ such as for every $g \in H(\Omega)$ with $\rho(f, g) < \varepsilon$ it follows $g \in D(L_k, p, q, s)$. The proof is similar to the one of subsection $4.2$ with a few differences.

Since $f \in D(L_k, p, q, s)$, we have that $f \in D_{p, q}(\zeta)$ for every $\zeta \in L$. It follows that $g \in D_{p, q}(\zeta)$ for every $\zeta \in L$ in the same way as we did in subsection $4.2$. By the triangle inequality, we have:
$$ \sup_{\zeta \in L} \sup_{z \in L_k} |[g; p / q]_{\zeta}(z) - g(z)| \leq \sup_{\zeta \in L} \sup_{z \in L_k} |[g; p / q]_{\zeta}(z) - [f; p / q]_{\zeta}(z)| + $$
$$ + \sup_{\zeta \in L} \sup_{z \in L_k} |[f; p / q]_{\zeta}(z) - f(z)| + \sup_{z \in L_k} |f(z) - g(z)| \;\;\; (4.4.1) $$

As we did in subsection $4.2$, for $\varepsilon > 0$ small, one obtains that the quantity $\sup_{z \in L_k} |f(z) - g(z)|$ can become arbitrary small, while at the same time the quantity $\sup_{\zeta \in L} \sup_{z \in L_k} |[g; p / q]_{\zeta}(z) - [f; p / q]_{\zeta}(z)|$ can become strictly less than $\frac{1}{s} - \sup_{\zeta \in L} \sup_{z \in L_k} |[f; p / q]_{\zeta}(z) - f(z)| > 0$, i. e.:
$$ \sup_{\zeta \in L} \sup_{z \in L_k} |[g; p / q]_{\zeta}(z) - [f; p / q]_{\zeta}(z)| < \frac{1}{s} - \sup_{\zeta \in L} \sup_{z \in L_k} |[f; p / q]_{\zeta}(z) - f(z)| \;\;\; (4.4.2) $$

The desired relation $\sup_{z \in L_k} |[g; p / q]_{\zeta}(z) - g(z)| < \frac{1}{s}$ follows by combining relations $(4.4.1)$ and $(4.4.2)$ and by using the triangle inequality.

It follows that $g \in D(L_k, p, q, s)$ and so $D(L_k, p, q, s)$ is open in $H(\Omega)$. The proof of Claim \ref{claim 4.4} is complete.

\subsection{Density of \boldmath{$\mathcal{X}(m, k, j, s)$}.} 

In order to prove Claim \ref{claim 4.5}, we fix the parameters $m, k, j, s \geq 1$ and we want to prove that the set:
$$ \mathcal{X}(m, k, j, s) = $$
$$ = \bigcup_{(p, q) \in F} A(K_m, p, j, s) \cap B(K_m, p, q, j, s) \cap C(L_k, p, s) \cap D(L_k, p, q, s) $$

is dense in $H(\Omega)$.

We consider a function $f \in H(\Omega)$, $L_1 \subseteq \Omega$ a compact set and $\varepsilon > 0$. We want to find a function $g \in \mathcal{X}(m, k, j, s)$ such that $ \sup_{z \in L_1} |f(z) - g(z)| < \varepsilon$.

We consider a compact set $L_{n_0} \subseteq \Omega$ with connected complement such that $L_1 \cup L_k \cup L \subseteq L_{n_0}$. Since $L_{n_0}$ and $K_m$ are disjoint compact sets with connected complements, the set $L_{n_0} \cup K_m$ is also a compact with connected complement. 

Consider now the following function:
$$ w(z) = \begin{cases} f_j(z), \;\; \mbox{if} \;\; z \in K_m 
\\ 
f(z),\;\; \mbox{if} \;\; z \in L_{n_0} \end{cases} $$

The function $w$ is well defined because $K_m \cap L_{n_0} = \emptyset$ and also $w \in A(K_M \cup L_{n_0})$. We apply Mergelyan's theorem to approximate $w$ by a polynomial $P$ uniformly on $K_m \cup L_{n_0}$. Our assumption on $F$ allows us to find an index $k_n \in \mathbb{N}$ such that $(p_{k_n}, q_{k_n}) \in F$ and $p_{k_n} > deg(P(z))$.

Now, let us consider the function $u(z) = P(z) + dz^{p_{k_n}}$, where $d \in \mathbb{C} \setminus \{ 0 \}$ and $|d|$ is small enough. It follows that $u$ and $w$ are uniformly close on $K_m \cup L_{n_0}$. This also implies that the functions $u$ and $f$ are uniformly close on $L_{n_0}$. Moreover, $u \in D_{p_{k_n}, q_{k_n}}(\zeta)$ and $[u; p_{k_n} / q_{k_n}]_{\zeta}(z) = u(z)$ for every $z, \zeta \in \mathbb{C}$ since $u$ is a polynomial. In order to prove that $u \in \mathcal{X}(m, k, j, s)$, we verify the following:

\begin{itemize}

\item[(1)]
$u \in D_{p_{k_n}, q_{k_n}}(\zeta)$ for every $\zeta \in \mathbb{C}$, since $u$ is a polynomial of degree exactly $p_{k_n}$, according to Proposition \ref{proposition 2.4}.

\item[(2)]
$u \in A(K_m, p_{k_n}, j, s)$ since:
$$ \sup_{\zeta \in L} \sup_{z \in K_m} |S_{p_{k_n}}(u, \zeta)(z) -  f_j(z)| = \sup_{z \in K_m} |u(z) -  f_j(z)| $$

The last quantity can become arbitrary small, provided that $|d| > 0$ is small enough.

\item[(3)]
$u \in B(K_m, p_{k_n}, q_{k_n}, j, s)$ since:
$$ \sup_{\zeta \in L} \sup_{z \in K_m} |[u; p_{k_n} / q_{k_n}]_{\zeta}(z) - f_j(z)| = \sup_{z \in K_m} |u(z) - f_j(z)| $$

The last quantity can become arbitrary small, provided that $|d| > 0$ is also small.

\item[(4)]
$u \in C(L_k, p_{k_n}, s)$ since:
$$ \sup_{\zeta \in L} \sup_{z \in L_k} |S_{p_{k_n}}(u, \zeta)(z) - u(z)| = \sup_{z \in L_k} |u(z) - u(z)| = 0 $$

\item[(5)]
$u \in D(L_k, p_{k_n}, q_{k_n}, s)$ since:
$$ \sup_{\zeta \in L} \sup_{z \in L_k} |[u; p_{k_n} / q_{k_n}]_{\zeta}(z) - u(z)| = \sup_{z \in L_k} |u(z) - u(z)| = 0 $$

\end{itemize}

It follows that $u \in \mathcal{X}(m, k, j, s)$. The proof of Claim \ref{claim 4.5} is complete. Baire's theorem yields the result.

\end{proof}

\begin{theorem} \label{theorem 4.6}

Let $\Omega \subseteq \mathbb{C}$ a simply connected domain and $\zeta \in \Omega$ be a fixed element. Also, let $F \subseteq \mathbb{N} \times \mathbb{N}$ be a non empty set containing exactly a sequence $(p_n, q_n)_{n \geq 1}$ such that $p_n \to + \infty$. Then, there exists a function $f \in H(\Omega)$ satisfying the following:

For every compact set $K \subseteq \mathbb{C} \setminus \Omega$ with connected complement and for every function $h \in A(K)$, there exists a subsequence $(p_{k_n}, q_{k_n})_{n \geq 1}$ of the sequence $(p_n, q_n)_{n \geq 1}$ such that:
\begin{itemize}

\item[(1)]
$f \in D_{p_{k_n}, q_{k_n}}(\zeta)$ for every $n \in \mathbb{N}$

\item[(2)]
$\sup_{z \in K} |S_{p_{k_n}}(f, \zeta)(z) - h(z)| \to 0$ as $n \to + \infty$

\item[(3)]
$\sup_{z \in K} |[f; p_{k_n} / q_{k_n}]_{\zeta}(z) - h(z)| \to 0$ as $n \to + \infty$

\item[(4)]
For every compact set $J \subseteq \Omega$ it holds:
$$ \sup_{z \in J} |S_{p_{k_n}}(f, \zeta)(z) - f(z)| \to 0 \; \text{as} \; n \to + \infty $$

\item[(5)]
For every compact set $J \subseteq \Omega$ it holds:
$$  \sup_{z \in J} |[f; p_{k_n} / q_{k_n}]_{\zeta}(z) - f(z)| \to 0 \; \text{as} \; n \to + \infty $$

\end{itemize}  

Moreover, the set of all functions satisfying $(1) - (5)$ is dense and $G_\delta$ in $H(\Omega)$.

\end{theorem}

\begin{proof} It suffices to apply Theorem \ref{theorem 4.3} for $L = \{ \zeta \}$.

\end{proof}

\begin{theorem} \label{theorem 4.7}

Let $\Omega \subseteq \mathbb{C}$ a simply connected domain. Also, let $F \subseteq \mathbb{N} \times \mathbb{N}$ be a non empty set containing exactly a sequence $(p_n, q_n)_{n \geq 1}$ such that $p_n \to + \infty$. Then, there exists a function $f \in H(\Omega)$ satisfying the following:

For every compact set $K \subseteq \mathbb{C} \setminus \Omega$ with connected complement and for every function $h \in A(K)$, there exists a subsequence $(p_{k_n}, q_{k_n})_{n \geq 1}$ of the sequence $(p_n, q_n)_{n \geq 1}$ such that for every compact set $L \subseteq \Omega$ we have:

\begin{itemize}

\item[(1)]
$f \in D_{p_{k_n}, q_{k_n}}(\zeta)$ for every $\zeta \in L$ and for every $n \geq n(L)$ for an index $n(L) \in \mathbb{N}$

\item[(2)]
$\sup_{\zeta \in L} \sup_{z \in K} |S_{p_{k_n}}(f, \zeta)(z) - h(z)| \to 0$ as $n \to + \infty$

\item[(3)]
$\sup_{\zeta \in L} \sup_{z \in K} |[f; p_{k_n} / q_{k_n}]_{\zeta}(z) - h(z)| \to 0$ as $n \to + \infty$

\item[(4)]
$\sup_{\zeta \in L} \sup_{z \in L} |S_{p_{k_n}}(f, \zeta)(z) - f(z)| \to 0$ as $n \to + \infty$

\item[(5)]
$\sup_{\zeta \in L} \sup_{z \in L} |[f; p_{k_n} / q_{k_n}]_{\zeta}(z) - f(z)| \to 0$ as $n \to + \infty$

\end{itemize}  

Moreover, the set of all functions satisfying $(1) - (5)$ is dense and $G_\delta$ in $H(\Omega)$.

\end{theorem}

\begin{proof} Let $\mathcal{A}$ denote the set of all functions satisfying $(1) - (5)$. We apply Theorem \ref{theorem 4.3} for $L = L_k$ and for $K = K_m$ (and that for every $k, m \geq 1$) and so, according to Baire's theorem we obtain a $G_{\delta}$ dense set in $H(\Omega)$; the set $\mathcal{A}_{k, m}$. One can verify by using Baire's theorem once more that $\mathcal{A} = \cap_{k, m \geq 1} \mathcal{A}_{k, m}$ and so $\mathcal{A}$ is also $G_{\delta}$ dense in $H(\Omega)$.

\end{proof}

\begin{remark} \label{remark 4.8}

The previous approximation (see for instance $(2)$ and $(3)$ of Theorem \ref{theorem 4.7}) may be strengthen to be valid for all order of derivatives, provided that the function $h$ is a polynomial. The new class of universal function is included in the old one and it is an open question if the inclusion is strict or not. The first paper where this has been done for Taylor series is \cite{dekaokto}. If the compact set $K$ where disjoint from $\overline{\Omega}$, then the approximation at the level of each derivative is automatic. We do not insist towards this direction with the exception of section $8$ below.

\end{remark}

\section{Universality in the sense of Luh and Chui - Parnes} 

If we replace the sets $\{ K_m \}_{m \geq 1}$ of section $4$ with the compact sets given below we obtain similar results in $H(\Omega)$ where the universal approximation is not requested to be valid on the boundary of $\Omega$.

\begin{lemma} \label{lemma 5.1} (\cite{dekatessera}, \cite{dekapente}) Let $\Omega \subseteq \mathbb{C}$ be a simply connected domain with $\overline{\Omega} \neq \mathbb{C}$. Then there exists a sequence of compact subsets $\{ K_m \}_{m \geq 1}$ of $\mathbb{C}$ with connected complements satisfying the following properties:  

\begin{itemize}

\item[(i)]
$K_m \cap \overline{\Omega} = \emptyset$ for every $m \geq 1$

\item[(ii)]
If $K \subseteq \mathbb{C}$ is a compact set with connected complement satisfying $K \cap \overline{\Omega} = \emptyset$, then there exists an index $m \in \mathbb{N}$ such that $K \subseteq K_m$

\end{itemize}

\end{lemma}

In this way we obtain $G_{\delta}$ - dense classes of functions which are larger than the classes of functions studied in section $4$ (see \cite{dio}, \cite{dekaepta}). This classes extend the classes of Universal Taylor series obtained in the 70's by Luh (\cite{dekatessera}, \cite{dekapente}) and Chui - Parnes (\cite{dekaexi}). The results we obtain are the following and we omit their proofs, since they are similar to the ones of the corresponding theorems in section $4$.

\begin{theorem} \label{theorem 5.2}

Let $\Omega \subseteq \mathbb{C}$ a simply connected domain with $\overline{\Omega} \neq \mathbb{C}$ and $L \subseteq \Omega$ be a compact set. Also, let $F \subseteq \mathbb{N} \times \mathbb{N}$ be a non empty set containing exactly a sequence $(p_n, q_n)_{n \geq 1}$ such that $p_n \to + \infty$. Then, there exists a function $f \in H(\Omega)$ satisfying the following:

For every compact set $K \subseteq \mathbb{C} \setminus \overline{\Omega}$ with connected complement and for every function $h \in A(K)$, there exists a subsequence $(p_{k_n}, q_{k_n})_{n \geq 1}$ of the sequence $(p_n, q_n)_{n \geq 1}$ such that:

\begin{itemize}

\item[(1)]
$f \in D_{p_{k_n}, q_{k_n}}(\zeta)$ for every $n \in \mathbb{N}$ and for every $\zeta \in L$

\item[(2)]
$\sup_{\zeta \in L} \sup_{z \in K} |S_{p_{k_n}}(f, \zeta)(z) - h(z)| \to 0$ as $n \to + \infty$

\item[(3)]
$\sup_{\zeta \in L} \sup_{z \in K} |[f; p_{k_n} / q_{k_n}]_{\zeta}(z) - h(z)| \to 0$ as $n \to + \infty$

\item[(4)]
For every compact set $J \subseteq \Omega$ it holds:
$$ \sup_{\zeta \in L} \sup_{z \in J} |S_{p_{k_n}}(f, \zeta)(z) - f(z)| \to 0 \; \text{as} \; n \to + \infty $$

\item[(5)]
For every compact set $J \subseteq \Omega$ it holds:
$$ \sup_{\zeta \in L} \sup_{z \in J} |[f; p_{k_n} / q_{k_n}]_{\zeta}(z) - f(z)| \to 0 \; \text{as} \; n \to + \infty $$
\end{itemize}  

Moreover, the set of all functions satisfying $(1) - (5)$ is dense and $G_\delta$ in $H(\Omega)$.

\end{theorem}

\begin{theorem} \label{theorem 5.3}

Let $\Omega \subseteq \mathbb{C}$ a simply connected domain and $\zeta \in \Omega$ be a fixed element. Also, let $F \subseteq \mathbb{N} \times \mathbb{N}$ be a non empty set containing exactly a sequence $(p_n, q_n)_{n \geq 1}$ such that $p_n \to + \infty$. Then, there exists a function $f \in H(\Omega)$ satisfying the following:

For every compact set $K \subseteq \mathbb{C} \setminus \overline{\Omega}$ with connected complement and for every function $h \in A(K)$, there exists a subsequence $(p_{k_n}, q_{k_n})_{n \geq 1}$ of the sequence $(p_n, q_n)_{n \geq 1}$ such that:

\begin{itemize}

\item[(1)]
$f \in D_{p_{k_n}, q_{k_n}}(\zeta)$ for every $n \in \mathbb{N}$

\item[(2)]
$\sup_{z \in K} |S_{p_{k_n}}(f, \zeta)(z) - h(z)| \to 0$ as $n \to + \infty$

\item[(3)]
$\sup_{z \in K} |[f; p_{k_n} / q_{k_n}]_{\zeta}(z) - h(z)| \to 0$ as $n \to + \infty$

\item[(4)]
For every compact set $J \subseteq \Omega$ it holds:
$$ \sup_{z \in J} |S_{p_{k_n}}(f, \zeta)(z) - f(z)| \to 0 \; \text{as} \; n \to + \infty $$

\item[(5)]
For every compact set $J \subseteq \Omega$ it holds:
$$  \sup_{z \in J} |[f; p_{k_n} / q_{k_n}]_{\zeta}(z) - f(z)| \to 0 \; \text{as} \; n \to + \infty $$

\end{itemize}  

Moreover, the set of all functions satisfying $(1) - (5)$ is dense and $G_\delta$ in $H(\Omega)$.

\end{theorem}

\begin{theorem} \label{theorem 5.4}

Let $\Omega \subseteq \mathbb{C}$ a simply connected domain. Also, let $F \subseteq \mathbb{N} \times \mathbb{N}$ be a non empty set containing exactly a sequence $(p_n, q_n)_{n \geq 1}$ such that $p_n \to + \infty$. Then, there exists a function $f \in H(\Omega)$ satisfying the following:

For every compact set $K \subseteq \mathbb{C} \setminus \overline{\Omega}$ with connected complement and for every function $h \in A(K)$, there exists a subsequence $(p_{k_n}, q_{k_n})_{n \geq 1}$ of the sequence $(p_n, q_n)_{n \geq 1}$ such that for every compact set $L \subseteq \Omega$ we have:
\begin{itemize}

\item[(1)]
$f \in D_{p_{k_n}, q_{k_n}}(\zeta)$ for every $\zeta \in L$ and for every $n \geq n(L)$ for an index $n(L) \in \mathbb{N}$

\item[(2)]
$\sup_{\zeta \in L} \sup_{z \in K} |S_{p_{k_n}}(f, \zeta)(z) - h(z)| \to 0$ as $n \to + \infty$

\item[(3)]
$\sup_{\zeta \in L} \sup_{z \in K} |[f; p_{k_n} / q_{k_n}]_{\zeta}(z) - h(z)| \to 0$ as $n \to + \infty$

\item[(4)]
$\sup_{\zeta \in L} \sup_{z \in L} |S_{p_{k_n}}(f, \zeta)(z) - f(z)| \to 0$ as $n \to + \infty$

\item[(5)]
$\sup_{\zeta \in L} \sup_{z \in L} |[f; p_{k_n} / q_{k_n}]_{\zeta}(z) - f(z)| \to 0$ as $n \to + \infty$

\end{itemize}  

Moreover, the set of all functions satisfying $(1) - (5)$ is dense and $G_\delta$ in $H(\Omega)$.

\end{theorem}

\begin{remark} \label{remark 5.5}

The universal functions in the sense of Luh and Chui - Parnes may be smooth on the boundary of $\Omega$. This is done in the following sections. We also mention that the universal functions in the sense of section $4$ can not be smooth on the boundary (\cite{dio}, \cite{tria}, \cite{tessera}, \cite{dekaepta}); thus the classes in section $4$ are strictly included in the corresponding classes of section $5$.

\end{remark}

\section{Universality in \boldmath{$A(\Omega)$}} 

In this section we combine and strengthen the results of \cite{dekaepta} and \cite{ennia}. We recall the following well known lemmas.

\begin{lemma} \label{lemma 6.1} (\cite{dekatessera}, \cite{dekapente}, Same as Lemma \ref{lemma 5.1}) Let $\Omega \subseteq \mathbb{C}$ be a simply connected domain with $\overline{\Omega} \neq \mathbb{C}$. Then there exists a sequence of compact subsets $\{ K_m \}_{m \geq 1}$ of $\mathbb{C}$ with connected complements satisfying the following properties:  

\begin{itemize}

\item[(i)]
$K_m \cap \overline{\Omega} = \emptyset$ for every $m \geq 1$

\item[(ii)]
If $K \subseteq \mathbb{C}$ is a compact set with connected complement satisfying $K \cap \overline{\Omega} = \emptyset$, then there exists an index $m \in \mathbb{N}$ such that $K \subseteq K_m$

\end{itemize}

\end{lemma}

\begin{lemma} \label{lemma 6.2}

Let $\Omega$ be an open set in $\mathbb{C}$ such that $\tilde{\mathbb{C}} \setminus \overline{\Omega}$ is connected. Then there exists a sequence $\{ L_k \}_{k \geq 1}$ of compact subsets of $\overline{\Omega}$ with connected complements such that:

\begin{itemize}

\item[(1)]
$L_k \subseteq L_{k + 1}$ for every $k \in \mathbb{N}$.

\item[(2)]
For every compact set $L \subseteq \overline{\Omega}$ there exists a $k \in \mathbb{N}$ such that $L \subseteq L_k$.  
\end{itemize}

For the proof of Lemma \ref{lemma 6.2} it suffices to set $L_k = \overline{\Omega} \cap \overline{B}(0, k)$ for every $k \in \mathbb{N}$.

\end{lemma}

We recall that $A(\Omega)$ contains exactly all functions $f: \overline{\Omega} \to \mathbb{C}$ which are continuous on $\overline{\Omega}$ and holomorphic on $\Omega$. The topology of $A(\Omega)$ is is defined by the seminorms $p_n(f) = \sup_{z \in \overline{\Omega} \cap \overline{B}(0, n)} |f(z)|$ for every $n \geq 1$. Thus, $A(\Omega)$ is a Fr\'{e}chet space and Baire's theorem is at our disposal.

We now present the main theorem of this section.

\begin{theorem} \label{theorem 6.3}

Let $\Omega \subseteq \mathbb{C}$ be a simply connected domain, such that $\overline{\Omega}^o = \Omega$, $\tilde{\mathbb{C}} \setminus \overline{\Omega}$ is connected and let $L \subseteq \Omega$ be a compact set. Also, let $F \subseteq \mathbb{N} \times \mathbb{N}$ be a non empty set containing exactly a sequence $(p_n, q_n)_{n \geq 1}$ such that $p_n \to + \infty$. Then, there exists a function $f \in A(\Omega)$ satisfying the following:

For every compact set $K$ with connected complement such that $K \cap \overline{\Omega} = \emptyset$ and for every $h \in A(K)$, there exists a subsequence $(p_{k_n}, q_{k_n})_{n \geq 1}$ of the sequence $(p_n, q_n)_{n \geq 1}$ such that:
\begin{itemize}

\item[(1)]
$f \in D_{p_{k_n}, q_{k_n}}(\zeta)$ for every $n \in \mathbb{N}$ and for every $\zeta \in L$

\item[(2)]
$\sup_{\zeta \in L} \sup_{z \in K} |S_{p_{k_n}}(f, \zeta)(z) - h(z)| \to 0$ as $n \to + \infty$

\item[(3)]
$\sup_{\zeta \in L} \sup_{z \in K} |[f; p_{k_n} / q_{k_n}]_{\zeta}(z) - h(z)| \to 0$ as $n \to + \infty$

\item[(4)]
For every compact set $J \subseteq \overline{\Omega}$ it holds:
$$ \sup_{\zeta \in L} \sup_{z \in J} |S_{p_{k_n}}(f, \zeta)(z) - f(z)| \to 0 \; \text{as} \; n \to + \infty $$

\item[(5)]
For every compact set $J \subseteq \overline{\Omega}$ it holds:
$$ \sup_{\zeta \in L} \sup_{z \in J} |[f; p_{k_n} / q_{k_n}]_{\zeta}(z) - f(z)| \to 0 \; \text{as} \; n \to + \infty $$

\end{itemize}  

Moreover, the set of all functions satisfying $(1) - (5)$ is dense and $G_\delta$ in $A(\Omega)$.

\end{theorem}

\begin{proof} Let us first consider an enumeration $\{ f_j \}_{j \geq 1}$ of polynomials with coefficients in $\mathbb{Q} + i\mathbb{Q}$.

Now, for every $(p, q) \in F$ and for every $j, s, m, k \geq 1$ we consider the following sets:
$$ A(K_m, p, j, s) = \{ f \in A(\Omega) : \sup_{\zeta \in L} \sup_{z \in K_m} |S_p(f,\zeta)(z) - f_j(z)| < \frac{1}{s}  \} $$
$$ B(K_m, p, q, j, s) = \{ f \in A(\Omega) : \; \text{for every} \; \zeta \in L \; \text{it holds} \; f \in D_{p,q}(\zeta) $$
$$ \text{and} \; \sup_{\zeta \in L} \sup_{z \in K_m} |[f; p / q]_{\zeta}(z) - f_j(z)| < \frac{1}{s} \} $$
$$ C(L_k, p, s) = \{ f \in A(\Omega) : \sup_{\zeta \in L} \sup_{z \in L_k} |S_p(f, \zeta)(z) - f(z)| < \frac{1}{s} \} $$
$$ D(L_k, p, q, s) = \{ f \in A(\Omega) : \; \text{for every} \; \zeta \in L \; \text{it holds} \; f \in D_{p,q}(\zeta) $$
$$ \text{and} \; \sup_{\zeta \in L} \sup_{z \in L_k} |[f; p / q]_{\zeta}(z) - f(z)| < \frac{1}{s} \} $$

One can verify (by using Mergelyan's theorem) that the set of all functions satisfying $(1) - (5)$ is precisely the following class:

$$ \mathcal{Y} = \bigcap_{j, m, k, s \geq 1} 
\bigcup_{(p, q) \in F} A(K_m, p, j, s) \cap B(K_m, p, q, j, s) \cap C(L_k, p, s) \cap  D(L_k, p, q, s) $$

So, according to Baire's theorem it is enough to prove the following:

\begin{claim} \label{claim 6.4}
\sloppy{The sets $A(K_m, p, j, s)$, $B(K_m, p, q, j, s)$, $C(L_k, p, s)$ and} $D(L_k, p, q, s)$ are open subsets of $A(\Omega)$ for every $(p, q) \in F$ and for every $m, k, j, s \geq 1$.
\end{claim}

\begin{claim} \label{claim 6.5}
The set:
$$ \mathcal{Y}(m, k, j, s) = \bigcup_{(p, q) \in F} A(K_m, p, j, s) \cap B(K_m, p, q, j, s) \cap C(L_k, p, s) \cap  D(L_k, p, q, s) $$ 

is a dense subset of $A(\Omega)$ for every $m, k, j, s \geq 1$.
\end{claim}

Since $A(\Omega)$ is a Fr\'{e}chet space, Baire's theorem implies that $\mathcal{Y}$ is a dense and $G_{\delta}$ subset of $A(\Omega)$.

For $q = 0$ we automatically have that for every function $f \in A(\Omega)$ it is $f \in D_{p, 0}(\zeta)$ for every $\zeta \in L$ and for every $p \geq 0$; in that case $[f; p / 0]_{\zeta} = S_p(f, \zeta)$. So, we restrict our attention to the case $q \geq 1$.

\subsection{The case of \boldmath{$A(K_m, p, j, s)$.}} 

We fix the parameters $(p, q) \in F$ and let $m, j, s \geq 1$. Let also $f \in A(K_m, p, j, s)$. We want to find a finite set $\emptyset \neq \Delta \subseteq \mathbb{N}^*$ and $\varepsilon > 0$ such that for every $g \in A(\Omega)$ with $g \in B_{\Delta}(f, \varepsilon)$ we obtain that $g \in A(K_m, p, j, s)$. Here:
$$ B_{\Delta}(f, \varepsilon) = \{ g \in A(\Omega) : \sup_{z \in L_k} |f(z) - g(z)| < \varepsilon \; \text{for every} \; k \in \Delta \}. $$

Since $B_{\{ max \; \Delta \}}(f, \varepsilon) = B_{\Delta}(f, \varepsilon)$, it suffices to consider $\Delta$ to be a singleton.

We know that $L$ is a compact set and thus the distance $d \equiv d(L, \mathbb{C} \setminus \Omega)$ is strictly positive. Hence, for every $\zeta \in L$ it is $B(\zeta, \frac{d}{2}) \subseteq \overline{B}(\zeta, \frac{d}{2}) \subseteq \Omega$. By compactness we can find an index $N \in \mathbb{N}$ and $\zeta_1, \cdots, \zeta_N \in L$ such that:
$$ L \subseteq \bigcup_{n = 1}^{N} B(\zeta_n, \frac{d}{2}) \subseteq \bigcup_{n = 1}^{N} \overline{B}(\zeta_n, \frac{d}{2}) \subseteq \Omega $$

From Lemma \ref{lemma 6.2} there exists an index $n_0 \in \mathbb{N}$ with $n_0 > p + 1$ such that $\bigcup_{n = 1}^{N} \overline{B}(\zeta_n, \frac{d}{2}) \subseteq L_{n_0}$ while at the same time the quantity $\sup_{\zeta\in L_{n_0}} |f(\zeta) - g(\zeta)|$ is small enough (provided that $\varepsilon > 0$ is also small).

By the Cauchy estimates, the first $p + 1$ Taylor coefficients of $g$ with center $\zeta \in L$ are uniformly close one by one to the corresponding Taylor coefficients of $f$. By the triangle inequality, we have:
$$ \sup_{\zeta \in L} \sup_{z \in K_m} |S_p(g, \zeta)(z) - f_j(z)| \} \leq $$
$$ \leq \sup_{\zeta \in L} \sup_{z \in K_m} |S_p(g, \zeta)(z) - S_p(f, \zeta)(z)| + \sup_{\zeta \in L} \sup_{z \in K_m} |S_p(f, \zeta)(z) - f_j(z)| \;\;\; (6.1.1)$$

Since $\frac{1}{s} - \sup_{\zeta \in L} \sup_{z \in K_m} |S_p(f, \zeta)(z) - f_j(z)| > 0$, we can obtain for $\varepsilon > 0$ small enough the following relation:
$$ \sup_{\zeta \in L} \sup_{z \in K_m} |S_p(g, \zeta)(z) - S_p(f, \zeta)(z)| < \frac{1}{s} - \sup_{\zeta \in L} \sup_{z \in K_m} |S_p(f, \zeta)(z) - f_j(z)| \;\;\; (6.1.2) $$ 

By combining relations $(6.1.1)$ and $(6.1.2)$ we aquire: 
$$ \sup_{\zeta \in L} \sup_{z \in K_m} |S_p(g, \zeta)(z) - f_j(z)| < \frac{1}{s} $$

It follows that $g \in A(K_m, p, j, s)$ and so $A(K_m, p, j, s)$ is open in $A(\Omega)$.

\subsection{The case of \boldmath{$B(K_m, p, q, j, s)$.}} 

\sloppy{We fix the parameters $(p, q) \in F$ and let $m, j, s \geq 1$. Let also $f \in B(K_m, p, q, j, s)$.} We want to find a finite set $\emptyset \neq \Delta \subseteq \mathbb{N}^*$ and $\varepsilon > 0$ such that for every $g \in A(\Omega)$ with $g \in B_{\Delta}(f, \varepsilon)$ we obtain that $g \in B(K_m, p, q, j, s)$. Once again, it suffices to consider $\Delta$ to be a singleton.

Since $f \in B(K_m, p, q, j, s)$ we have that $f \in D_{p, q}(\zeta)$ for every $\zeta \in L$. Moreover, the Hankel determinant $D_{p, q}(f, \zeta)$ for $f$ is not equal to $0$ and that holds for every $\zeta \in L$. By continuity, there exists a $\delta > 0$ such that:
$$ |D_{p, q}(f, \zeta)| > \delta \; \text{for every} \; \zeta \in L $$

Suppose that the first $p + q + 1$ Taylor coefficients of $g$ are uniformly close enough one by one to the corresponding Taylor coefficients of $f$, provided that $\varepsilon > 0$ is small. Again by continuity, we obtain:
$$ |D_{p, q}(g, \zeta)| > \frac{\delta}{2} \; \text{for every} \; \zeta \in L $$

The last relation implies that $g \in D_{p, q}(\zeta)$ for every $\zeta \in L$. By the triangle inequality, we have:
$$ \sup_{\zeta \in L} \sup_{z \in K_m} |[g; p / q]_{\zeta}(z) - f_j(z)| \leq $$
$$ \leq \sup_{\zeta \in L} \sup_{z \in K_m} |[g; p / q]_{\zeta}(z) - [f; p / q]_{\zeta}(z)| + \sup_{\zeta \in L} \sup_{z \in K_m} |[f; p / q]_{\zeta}(z) - f_j(z)| \;\;\; (6.2.1) $$

We recall the Pad\'{e} approximants:
$$ [f; p / q]_{\zeta}(z) = \frac{A(f, \zeta)(z)}{B(f, \zeta)(z)} $$

and
$$ [g; p / q]_{\zeta}(z) = \frac{A(g, \zeta)(z)}{B(g, \zeta)(z)} $$

where the polynomials $A(f, \zeta)(z)$, $B(f, \zeta)(z)$, $A(g, \zeta)(z)$ and $B(g, \zeta)(z)$ are given by the Jacobi formulas and their coefficients vary continuously on $\zeta \in L$. Since $\sup_{\zeta \in L} \sup_{z \in K_m} |[f; p / q]_{\zeta}(z) - f_j(z)| < \frac{1}{s}$ we have that the function $[f; p / q]_{\zeta}(z)$ takes only finite values for every $z \in K_m$ and for every $\zeta \in L$. Thus, $B(f, \zeta)(z) \neq 0$ for every $\zeta \in L$ and for every $z \in K_m$. By continuity, one obtains that there exists a $\delta' >0$ such that:
$$ |B(f, \zeta)(z)| \geq \delta' \; \text{for every} \; \zeta \in L \; \text{and for every} \; z \in K_m $$ 
 
Since the first $p + q + 1$ Taylor coefficients of $f$ are uniformly close enough one by one to those of $g$ (provided that $\varepsilon > 0$ is small), one obtains:
$$ |B(g, \zeta)(z)| \geq \frac{\delta'}{2} \; \text{for every} \; \zeta \in L \; \text{and for every} \; z \in K_m $$ 

By the triangle inequality, it is:
$$ |[g; p / q]_{\zeta}(z) - [f; p / q]_{\zeta}(z)| = |\frac{A(f, \zeta)(z)}{B(f, \zeta)(z)} - \frac{A(g, \zeta)(z)}{B(g, \zeta)(z)}| = $$
$$ = |\frac{A(f, \zeta)(z)B(g, \zeta)(z) - A(g, \zeta)(z)B(f, \zeta)(z)}{B(f, \zeta)(z)B(g, \zeta)(z)}| \leq $$
$$ \leq \frac{2}{(\delta')^2}|A(f, \zeta)(z)B(g, \zeta)(z) - A(g, \zeta)(z)B(f, \zeta)(z)| \leq $$
$$ \leq \frac{2}{(\delta')^2}|A(f, \zeta)(z)||B(f, \zeta)(z) - B(g, \zeta)(z)| + $$
$$ + \frac{2}{(\delta')^2}|B(f, \zeta)(z)||A(f, \zeta)(z) - A(g, \zeta)(z)| $$

Thus, provided that $\varepsilon > 0$ is small, we obtain: 
$$ \sup_{\zeta \in L} \sup_{z \in K_m} |[g; p / q]_{\zeta}(z) - [f; p / q]_{\zeta}(z)| < $$
$$ < \frac{1}{s} - \sup_{\zeta \in L} \sup_{z \in K_m} |[f; p / q]_{\zeta}(z) - f_j(z)| \;\;\; (6.2.2) $$

By combining relations $(6.2.1)$ and $(6.2.2)$ we acquire: 
$$ \sup_{\zeta \in L} \sup_{z \in K_m} |[g; p / q]_{\zeta}(z) - f_j(z)| < \frac{1}{s} $$

It follows that $g \in B(K_m, p, q, j, s)$ and so $B(K_m, p, q, j, s)$ is open in $A(\Omega)$.

\subsection{The case of \boldmath{$C(L_k, p, s)$.}} 

We fix the parameters $(p, q) \in F$ and let $k, s \geq 1$. Let also $f \in C(L_k, p, s)$. We want to find a finite set $\emptyset \neq \Delta \subseteq \mathbb{N}^*$ and $\varepsilon > 0$ such that for every $g \in A(\Omega)$ with $g \in B_{\Delta}(f, \varepsilon)$ we obtain that $g \in C(L_k, p, s)$. The proof is similar to the one of subsection $6.1$ except from the following difference: 
$$ \sup_{\zeta \in L} \sup_{z \in L_k} |S_p(g, \zeta)(z) - g(z)| \leq \sup_{\zeta \in L} \sup_{z \in L_k} |S_p(g, \zeta)(z) - S_p(f, \zeta)(z)| + $$
$$ +\sup_{\zeta \in L} \sup_{z \in L_k} |S_p(f, \zeta)(z) - f(z)| + \sup_{z \in L_k} |f(z) - g(z)| \;\;\; (6.3.1) $$ 

If $\varepsilon > 0$ is small enough, one obtains that the quantity $\sup_{z \in L_k} |f(z) - g(z)|$ is also small. By the Cauchy estimates, we can achieve the following: 
$$ \sup_{\zeta \in L} \sup_{z \in L_k} |S_p(g, \zeta)(z) - S_p(f, \zeta)(z)| < \frac{1}{s} - \sup_{\zeta \in L} \sup_{z \in L_k} |S_p(f, \zeta)(z) - f(z)| - $$
$$ - \sup_{z \in L_k} |f(z) - g(z)| \;\;\; (6.3.2) $$

for $\varepsilon > 0$ small. By combining relations $(6.3.1)$ and $(6.3.2)$ we obtain: 
$$ \sup_{\zeta \in L} \sup_{z \in L_k} |S_p(g, \zeta)(z) - g(z)| < \frac{1}{s}. $$

It follows that $g \in C(L_k, p, s)$ and so $C(L_k, p, s)$ is open in $A(\Omega)$.

\subsection{The case of \boldmath{$D(L_k, p, q, s)$.}} 

We fix the parameters $(p, q) \in F$ and let $k, s \geq 1$. Let also $f \in D(L_k, p, q, s)$. We want to find a finite set $\emptyset \neq \Delta \subseteq \mathbb{N}^*$ and $\varepsilon > 0$ such that for every $g \in A(\Omega)$ with $g \in B_{\Delta}(f, \varepsilon)$ we obtain that $g \in D(L_k, p, q, s)$. The proof is similar to the one of subsection $6.2$ with a few differences.

Since $f \in D(L_k, p, q, s)$ we have that $f \in D_{p, q}(\zeta)$ for every $\zeta \in L$. It follows that $g \in D_{p, q}(\zeta)$ for every $\zeta \in L$ in the same way as we did in $(6.2)$.

By the triangle inequality, we have:
$$ \sup_{\zeta \in L} \sup_{z \in L_k} |[g; p / q]_{\zeta}(z) - g(z)| \leq \sup_{\zeta \in L} \sup_{z \in L_k} |[g; p / q]_{\zeta}(z) - [f; p / q]_{\zeta}(z)| + $$
$$ + \sup_{\zeta \in L} \sup_{z \in L_k} |[f; p / q]_{\zeta}(z) - f(z)| + \sup_{z \in L_k} |f(z) - g(z)| \;\;\; (6.4.1) $$

As we did in subsection $6.2$, for $\varepsilon > 0$ small enough, one obtains that the quantity $\sup_{z \in L_k} |f(z) - g(z)|$ is arbitrary small, while at the same time the quantity $\sup_{\zeta \in L} \sup_{z \in L_k} |[g; p / q]_{\zeta}(z) - [f; p / q]_{\zeta}(z)|$ can become strictly less than $\frac{1}{s} - \sup_{\zeta \in L} \sup_{z \in L_k} |[f; p / q]_{\zeta}(z) - f(z)| > 0$, i. e.:
$$ \sup_{\zeta \in L} \sup_{z \in L_k} |[g; p / q]_{\zeta}(z) - [f; p / q]_{\zeta}(z)| < \frac{1}{s} - \sup_{\zeta \in L} \sup_{z \in L_k} |[f; p / q]_{\zeta}(z) - f(z)| \;\;\;(6.4.2) $$

The desired relation ($\sup_{z \in L_k} |[g; p / q]_{\zeta}(z) - g(z)| < \frac{1}{s}$) follows by combining relations $(6.4.1)$ and $(6.4.2)$ and by using the triangle inequality.

It follows that $g \in D(L_k, p, q, s)$ and so $D(L_k, p, q, s)$ is open in $A(\Omega)$. The proof of Claim \ref{claim 6.4} is complete.

\subsection{Density of \boldmath{$\mathcal{Y}(m, k, j, s)$.}} 

In order to prove Claim \ref{claim 6.5} we fix the parameters $m, k, j, s \geq 1$ and we want to prove that the set:
$$ \mathcal{Y}(m, k, j, s) = $$
$$ = \bigcup_{(p, q) \in F} A(K_m, p, j, s) \cap B(K_m, p, q, j, s) \cap C(L_k, p, s) \cap  D(L_k, p, q, s) $$ 

is a dense subset of $A(\Omega)$. 

Let us consider a function $f \in A(\Omega)$, $\emptyset \neq \Delta \subseteq \mathbb{N}^*$ a finite set and $\varepsilon > 0$. We are looking for a function $g \in \mathcal{Y}(m, k, j, s)$ such that $g \in B_{\Delta}(f, \varepsilon)$. If $N = max \; \Delta$, it suffices to find a $g \in \mathcal{Y}(m, k, j, s)$ satisfying $p_{N}(f - g) < \varepsilon$ (since $p_N(h) = \sup_{\overline{\Omega} \cap \overline{B}(0, N)}|h(z)|$ for every function $h \in A(\Omega)$ and $B_{\Delta}(f, \varepsilon) = B_{max \{ \Delta \}}(f, \varepsilon)$).

We notice that $L_k \cup (\overline{\Omega} \cap \overline{B}(0, N))$ is a compact subset of $\overline{\Omega}$, so according to Lemma \ref{lemma 6.2}, there exists a $k_0 > k$ such that $L_k \cup (\overline{\Omega} \cap \overline{B}(0, N)) \subseteq L_{k_0}$. The compact sets $L_{k_0}$ and $K_m$ are disjoint with connected complements and thus the set $L_{k_0} \cup K_m$ is also a compact one with connected complement. 

One can apply Mergelyan's theorem for the following function:
$$ w(z) = \begin{cases} f_j(z), \;\; \mbox{if} \;\; z \in K_m 
\\ 
f(z),\;\; \mbox{if} \;\; z \in L_{k_0} \end{cases} $$

The function $w$ is well defined (because $K_m \cap L_{k_0} = \emptyset$) and also $w \in A(K_m \cup L_{k_0})$ since $\overline{\Omega}^o = \Omega$ and $f$ is holomorphic in ${L_{k_0}}^o$ for every $f \in A(\Omega)$.

The function $w$ can be uniformly approximated on $L_{k_0} \cup K_m$ by a polynomial $P$. Our assumption on $F$ allows us to find an index $k_n \in \mathbb{N}$ such that $(p_{k_n}, q_{k_n}) \in F$ and $p_{k_n} > deg(P(z))$. Then, the function $u(z) = P(z) + dz^{p_{k_n}}$, where $d \in \mathbb{C} \setminus \{ 0 \}$ and $0 < |d|$ is small enough, is clearly a polynomial that is uniformly close to the function $w(z)$ on $L_{k_0} \cup K_m$.

Moreover, according to Proposition \ref{proposition 2.4} we have that $u \in D_{p_{k_n}, q_{k_n}}(\zeta)$ and $[u; p_{k_n} / q_{k_n}]_{\zeta}(z) = u(z)$ for every $\zeta, z \in \mathbb{C}$; in particular for every $z \in L_{k_0} \cup K_m$. 

Now, one can easily check that the fuction $u$ not only satisfies $p_N(f - u) < \varepsilon$, but also that $u \in \mathcal{Y}(m, k, j, s)$, provided that $0 < |d|$ is small enough.

The proof of Claim \ref{claim 6.5} is complete. The result follows from Baire's theorem.

\end{proof}

\begin{theorem} \label{theorem 6.6}

Let $\Omega \subseteq \mathbb{C}$ be a simply connected domain, such that $\overline{\Omega}^o = \Omega$, $\tilde{\mathbb{C}} \setminus \overline{\Omega}$ is connected and let $\zeta \in \Omega$ be a fixed element. Also, let $F \subseteq \mathbb{N} \times \mathbb{N}$ be a non empty set containing exactly a sequence $(p_n, q_n)_{n \geq 1}$ such that $p_n \to + \infty$. Then, there exists a function $f \in A(\Omega)$ satisfying the following:

For every compact set $K \subseteq \mathbb{C} \setminus \overline{\Omega}$ with connected complement and for every function $h \in A(K)$, there exists a subsequence $(p_{k_n}, q_{k_n})_{n \geq 1}$ of the sequence $(p_n, q_n)_{n \geq 1}$ such that:

\begin{itemize}

\item[(1)]
$f \in D_{p_{k_n}, q_{k_n}}(\zeta)$ for every $n \in \mathbb{N}$

\item[(2)]
$\sup_{z \in K} |S_{p_{k_n}}(f, \zeta)(z) - h(z)| \to 0$ as $n \to + \infty$

\item[(3)]
$\sup_{z \in K} |[f; p_{k_n} / q_{k_n}]_{\zeta}(z) - h(z)| \to 0$ as $n \to + \infty$

\item[(4)]
For every compact set $J \subseteq \overline{\Omega}$ it holds:
$$ \sup_{z \in J} |S_{p_{k_n}}(f, \zeta)(z) - f(z)| \to 0 \; \text{as} \; n \to + \infty $$

\item[(5)]
For every compact set $J \subseteq \overline{\Omega}$ it holds:
$$ \sup_{z \in J} |[f; p_{k_n} / q_{k_n}]_{\zeta}(z) - f(z)| \to 0 \; \text{as} \; n \to + \infty $$

\end{itemize}  

Moreover, the set of all functions satisfying $(1) - (5)$ is dense and $G_\delta$ in $A(\Omega)$.

\end{theorem}

\begin{proof} It suffices to apply Theorem \ref{theorem 6.3} for $L = \{ \zeta \}$.

\end{proof}

\begin{theorem} \label{theorem 6.7}

Let $\Omega \subseteq \mathbb{C}$ a simply connected domain, such that $\overline{\Omega}^o = \Omega$. Also, let $F \subseteq \mathbb{N} \times \mathbb{N}$ be a non empty set containing exactly a sequence $(p_n, q_n)_{n \geq 1}$ such that $p_n \to + \infty$. Then, there exists a function $f \in A(\Omega)$ satisfying the following:

For every compact set $K \subseteq \mathbb{C} \setminus \Omega$ with connected complement and for every function $h \in A(K)$, there exists a subsequence $(p_{k_n}, q_{k_n})_{n \geq 1}$ of the sequence $(p_n, q_n)_{n \geq 1}$ such that for every compact set $L \subseteq \Omega$ we have:
\begin{itemize}

\item[(1)]
$f \in D_{p_{k_n}, q_{k_n}}(\zeta)$ for every $\zeta \in L$ and for every $n \geq n(L)$ for an index $n(L) \in \mathbb{N}$

\item[(2)]
$\sup_{\zeta \in L} \sup_{z \in K} |S_{p_{k_n}}(f, \zeta)(z) - h(z)| \to 0$ as $n \to + \infty$

\item[(3)]
$\sup_{\zeta \in L} \sup_{z \in K} |[f; p_{k_n} / q_{k_n}]_{\zeta}(z) - h(z)| \to 0$ as $n \to + \infty$

\item[(4)]
$\sup_{\zeta \in L} \sup_{z \in J} |S_{p_{k_n}}(f, \zeta)(z) - f(z)| \to 0$ as $n \to + \infty$ for every compact set $J \subseteq \overline{\Omega}$

\item[(5)]
$\sup_{\zeta \in L} \sup_{z \in J} |[f; p_{k_n} / q_{k_n}]_{\zeta}(z) - f(z)| \to 0$ as $n \to + \infty$ for every compact set $J \subseteq \overline{\Omega}$

\end{itemize}  

Moreover, the set of all functions satisfying $(1) - (5)$ is dense and $G_\delta$ in $A(\Omega)$.

\end{theorem}

\begin{proof} Let $\mathcal{A}$ denote the set of all functions satisfying Theorem \ref{theorem 6.7}. We apply Theorem \ref{theorem 6.3} for $L = L_k$ and for $K = K_m$ (and that for every $k, m \geq 1$) and so, according to Baire's theorem we obtain a $G_{\delta}$ - dense set in $H(\Omega)$; the set $\mathcal{A}_{k, m}$. One can verify by using Baire's theorem once more that $\mathcal{A} = \cap_{k, m \geq 1} \mathcal{A}_{k, m}$ and so $\mathcal{A}$ is also $G_{\delta}$ - dense in $A(\Omega)$.

\end{proof}

\section{Universality in a subspace of \boldmath{$A^{\infty}(\Omega)$}} 

In this section we combine and strengthen the results of \cite{deka}, \cite{eikosiena} and \cite{eikosidio}.

Let $\Omega \subseteq \mathbb{C}$ be an open set. We say that a holomorphic function $f$ defined on $\Omega$ belongs to $A^{\infty}(\Omega)$ if for every $l \in \mathbb{N}$ the $l$th derivative $f^{(l)}$ of $f$ extends continuously on $\overline{\Omega}$. In $A^{\infty}(\Omega)$ we consider the topology defined by the seminorms $\sup_{z \in L_k} |f^{(l)}(z)|$, for every $k \geq 1$ and for every $l \in \mathbb{N}$, where $\{ L_k \}_{k \geq 1}$ is a family of compact subsets of $\overline{\Omega}$ such that for every compact set $L \subseteq \overline{\Omega}$ there exists an index $k \in \mathbb{N}$ satisfying $L \subseteq L_k$. Such a family for example is obtained by setting $L_k = \overline{\Omega} \cap \overline{B}(0, k)$ for every $k \in \mathbb{N}$. With this topology, $A^{\infty}(\Omega)$ becomes a Fr\'{e}chet space.

We call $X^{\infty}(\Omega)$ the closure in $A^{\infty}(\Omega)$ of all the rational functions with poles off $\overline{\Omega}$.

We present now the main theorem of this section.

\begin{theorem} \label{theorem 7.1}

Let $F \subseteq \mathbb{N} \times \mathbb{N}$ be a non empty set that contains exactly a sequence $(p_n, q_n)_{n \geq 1}$ such that $p_n \to + \infty$. Also, let $\Omega \subseteq \mathbb{C}$ be an open set such that $\overline{\Omega}^{o} = \Omega$ and $\{ \infty \} \cup (\mathbb{C} \setminus \overline{\Omega})$ is connected. Consider $L \subseteq \overline{\Omega}$ a compact set and $K \subseteq \mathbb{C} \setminus \overline{\Omega}$ another compact set with connected complement. Then there exists a function $f \in X^{\infty}(\Omega)$ satisfying the following:

For every function $h \in A(K)$ there exists a subsequence $(p_{k_n}, q_{k_n})_{n \geq 1}$ of the sequence $(p_n, q_n)_{n \geq 1}$ such that:

\begin{itemize}

\item[(1)]
$f \in D_{p_{k_n}, q_{k_n}}(\zeta)$ for every $\zeta \in L$ and for every $n \geq 1$

\item[(2)]
For every $l \in \mathbb{N}$ it holds:
$$ \sup_{\zeta \in L} \sup_{z \in J} |{[f; p_{k_n} / q_{k_n}]^{(l)}_{\zeta}} - f^{(l)}(z)| \to 0 \; \text{as} \; n \to + \infty $$ 
 
for every compact set $J \subseteq \overline{\Omega}$.

\item[(3)]
$\sup_{\zeta \in L} \sup_{z \in K} |[f; p_{k_n} / q_{k_n}]_{\zeta} - h(z)| \to 0$ as $n \to + \infty$

\item[(4)]
$\sup_{\zeta \in L} \sup_{z \in K} |S_{p_{k_n}}(f, \zeta)(z) - h(z)| \to 0$ as $n \to \infty$

\item[(5)]
For every $l \in \mathbb{N}$ it holds:
$$ \sup_{\zeta \in L} \sup_{z \in J} |S_{p_{k_n}}^{(l)}(f, \zeta)(z) - f^{(l)}(z)| \to 0 \; \text{as} \; n \to + \infty $$

for every compact set $J \subseteq \overline{\Omega}$.

\end{itemize}

Moreover, the set of all functions satisfying $(1) - (5)$ is dense and $G_{\delta}$ in $X^{\infty}(\Omega)$. 

\end{theorem}

\begin{proof} Let $\{ f_j \}_{j \geq 1}$ be an enumeration of polynomials with coefficients in $\mathbb{Q} + i \mathbb{Q}$. Since $\{ \infty \} \cup (\mathbb{C} \setminus \overline{\Omega})$ is connected, the sequence $\{ f_j \}_{j \geq 1}$ is a dense subset of $X^{\infty}(\Omega)$.

We consider $\mathcal{U}$ to be the set of all functions in $X^{\infty}(\Omega)$ that satisfy $(1) - (5)$. Our aim is to prove that $\mathcal{U}$ is a dense and $G_{\delta}$ subset of $X^{\infty}(\Omega)$. 

Now, for every $(p, q) \in F$ and for every $j, s, n \geq 1$ we consider the following sets: 
$$ A(n, p, q, s) = \{ f \in X^{\infty}(\Omega) : f \in D_{p, q}(\zeta) \; \text{for every} \; \zeta \in L \; \text{and} $$ 
$$ \sup_{\zeta \in L} \sup_{z \in L_n} |[f; p / q]^{(l)}_{\zeta}(z) - f^{(l)}(z)| < \frac{1}{s} \; \text{for every} \; l = 0, 1, 2, \cdots, s \} $$
$$ B(p, q, j, s) = \{ f \in X^{\infty}(\Omega) : f \in D_{p, q}(\zeta) \; \text{for every} \; \zeta \in L \; \text{and} $$
$$ \sup_{\zeta \in L} \sup_{z \in K} |[f; p / q]_{\zeta}(z) - f_{j}(z)| < \frac{1}{s} \} $$
$$ E(p, j, s) = \{ f \in X^{\infty}(\Omega) : \sup_{\zeta \in L} \sup_{z \in K} |S_{p}(f, \zeta)(z) - f_{j}(z)| < \frac{1}{s} \} $$
$$ F(p, s, n) = \{ f \in X^{\infty}(\Omega) : \sup_{\zeta \in L} \sup_{z \in L_n} |S^{(l)}_{p}(f, \zeta)(z) - f^{(l)}(z)| < \frac{1}{s} $$
$$ \text{for every} \; l = 0, 1, 2, \cdots, s \}$$

One can verify that $\mathcal{U}$ is precisely the following class:
$$ \mathcal{U} = \bigcap_{n, j, s \geq 1} \bigcup_{(p, q) \in F} A(n, p, q, s) \cap B(p, q, j, s) \cap E(p, j, s) \cap F(p, s, n) $$

So, according to Baire's theorem, it is enough to prove the following:

\begin{claim} \label{claim 7.2}

The sets $A(n, p, q, s), B(p, q, j, s), E(p, j, s)$ and $F(p, s, n)$ are open subsets of $X^{\infty}(\Omega)$ for every $(p, q) \in F$ and for every $j, s, n \geq 1$.

\end{claim}

\begin{claim} \label{claim 7.3}

The set:
$$ \mathcal{U}(n, j, s) = $$
$$ = \bigcup_{(p, q) \in F} A(n, p, q, s) \cap B(p, q, j, s) \cap E(p, j, s) \cap F(p, s, n)$$ 

is a dense subset of $X^{\infty}(\Omega)$ for every $j, s, n \geq 1$.

\end{claim}

For $q = 0$ we automatically have that for every function $f \in X^{\infty}(\Omega)$ it holds $f \in D_{p, 0}(\zeta)$ for every $\zeta \in L$ and for every $p \geq 0$; we have in that case $[f; p / 0]_{\zeta} = S_p(f, \zeta)$. So, we restrict our attention to the case $q \geq 1$.

\subsection{The case of \boldmath{$A(n, p, q, s)$.}} 

We fix the parameters $(p, q) \in F$ and $n, s \geq 1$. Let $f \in A(n, p, q, s)$, $g \in X^{\infty}(\Omega)$ and $a > 0$ be small enough. The number $a > 0$ will be determined later on. We consider a compact set $L_{n_0} \subseteq \overline{\Omega}$ such that $L \cup L_n \subseteq L_{n_0}$. In addition, suppose that the following holds:
$$ \sup_{z \in L_{n_0}} |f^{(m)}(z) - g^{(m)}(z)| < a $$
$$ \text{for every} \; m = 0, 1, 2, \cdots, \; max(s, p + q + 1) \;\;\; (7.1.1) $$

We will show that if $a > 0$ is small enough we obtain that $g \in A(n, p, q, s)$.

Since $f \in A(n, p, q, s)$, the Hankel determinant $D_{p, q}(f, \zeta)$ for $f$ depends continuously on $\zeta \in L$; thus there exists a $\delta > 0$ such that $|D_{p, q}(f, \zeta)| > \delta$ for every $\zeta \in L$. From relation $(7.1.1)$ one derives that if $a > 0$ is  sufficiently small then the Hankel determinant $D_{p, q}(g, \zeta)$ for $g$ satisfies $|D_{p, q}(g, \zeta)| > \frac{\delta}{2}$ for every $\zeta \in L$. Therefore, $g \in D_{p, q}(\zeta)$ for every $\zeta \in L$.

On the other hand, $f(z) \in \mathbb{C}$ for every $z \in L_n$; so $[f; p / q]_{\zeta} \in \mathbb{C}$ for every $\zeta \in L$ and for every $z \in L_n$. It follows that $B(f, \zeta)(z) \neq 0$ for every $\zeta \in L$ and for every $z \in L_n$, where the polynomial $B(f, \zeta)(z)$ is defined according to the Jacobi formula. By continuity we have that there exists a $0 < \delta'' < 1$ such that $|B(f, \zeta)(z)| > \delta''$ for every $(\zeta, z) \in L \times L_n$ (since the function $B(f, \zeta)(z)$ is a continuous function on $L \times L_n$). For $a > 0$ sufficiently small, one can achieve the following: 
$$ |B(g, \zeta)(z)| > \frac{\delta''}{2} \; \text{for every} \; (\zeta, z) \in L \times L_n $$
 
Now, for every $l \in \{ 0, 1, 2, \cdots, s \}$ and from the triangle enequality, we obtain: 
$$ \sup_{\zeta \in L} \sup_{z \in L_n} |{[g; p / q]^{(l)}_{\zeta}}(z) - g^{(l)}(z)| \leq \sup_{\zeta \in L} \sup_{z \in L_n} |[g; p / q]^{(l)}_{\zeta}(z) - [f; p / q]^{(l)}_{\zeta}(z)| + $$
$$ + \sup_{\zeta \in L} \sup_{z \in L_n} |[f; p / q]_{\zeta}^{(l)}(z) - f^{(l)}(z)| + \sup_{z \in L_n} |f^{(l)}(z) - g^{(l)}(z)| \;\;\; (7.1.2) $$

The term $\sup_{z \in L_n} |f^{(l)}(z) - g^{(l)}(z)|$ can become as small as we want, provided that $a > 0$ is sufficiently small and because $L_n \subseteq L_{n_0}$.

The term $\sup_{\zeta \in L} \sup_{z \in L_n} |[f; p / q]_{\zeta}^{(l)}(z) - f^{(l)}(z)|$ is strictly less than $\frac{1}{s}$ since $f \in A(n, p, q, s)$. It remains to control the term:
$$ \sup_{\zeta \in L} \sup_{z \in L_n} |[g; p / q]^{(l)}_{\zeta}(z) - [f; p / q]^{(l)}_{\zeta}(z)| $$

The denominators of $[f; p / q]_{\zeta}^{(l)}$ and $[g; p / q]_{\zeta}^{(l)}$ are bounded below by $(\delta'')^{l + 1}$ and $(\frac{\delta''}{2})^{l + 1}$ respectively for every $l = 0, 1, 2, \cdots, s$. 

So, the quantity $\sup_{\zeta \in L} \sup_{z \in L_n} |[g; p / q]^{(l)}_{\zeta}(z) - [f; p / q]^{(l)}_{\zeta}(z)|$ can become arbitrary small, since the coefficients of the polynomials $A(f, \zeta)$ and $A(g, \zeta)$ are close enough one by one. The same happens for the polynomials $B(f, \zeta)$ and $B(g, \zeta)$. This allows us to control every finite set of derivatives. It follows that $g \in A(n, p, q, s)$ and thus $A(n, p, q, s)$ is an open set of $X^{\infty}(\Omega)$.

\subsection{The case of \boldmath{$B(p, q, j, s)$.}} 

We fix the parameters $(p, q) \in F$ and $j, s \geq 1$. Let $f \in B(p, q, j, s)$, $g \in X^{\infty}(\Omega)$ and $a > 0$ be small. The number $a > 0$ will be determined later on. Suppose that the following holds:
$$ \sup_{z \in L_{n_0}} |f^{(m)}(z) - g^{(m)}(z)| < a $$
$$ \text{for every} \; m = 0, 1, 2, \cdots, \; max(s, p + q + 1) \;\;\; (7.2.1) $$

We will show that if $a > 0$ is small enough we obtain that $g \in B(p, q, j, s)$.

In order to prove that $g \in D_{p, q}(\zeta)$ for every $\zeta \in L$ we follow exactly the same steps as we did in subsection $7.1$. This part of the proof is omitted.

Since $f \in B(p, q, j, s)$ we have that $|[f; p / q]_{\zeta}(z) - f_{j}(z)| < \frac{1}{s}$ for every $\zeta \in L$ and for every $z \in K$, so $[f; p / q]_{\zeta}(z) \in \mathbb{C}$ for every $\zeta \in L$ and for every $z \in K$. It follows that there exists a $\delta' > 0$ such that $|B(f, \zeta)(z)| > \delta'$ for every $\zeta \in L$ and for every $z \in K$, since $L \times K$ is a compact set, where $B(f, \zeta)$ is the denominator of the Jacobi formula for $f$. From relation $(7.2.1)$ and the Jacobi formula it follows that for $a > 0$ sufficiently small, it holds:
$$ |B(g, \zeta)(z)| > \frac{\delta'}{2} \; \text{for every} \; (\zeta, z) \in L \times K. $$

So, it suffices to prove the following:
$$ \sup_{\zeta \in L} \sup_{z \in K} |[g; p / q]_{\zeta}(z) - [f; p / q]_{\zeta}(z)| < $$
$$ < \frac{1}{s} - \sup_{\zeta \in L} \sup_{z \in K} |[f; p / q]_{\zeta}(z) - f_{j}(z)| \;\;\; (7.2.2) $$

and then the result follows from the triangle inequality. We also have: 
$$ \sup_{\zeta \in L} \sup_{z \in K} |[g; p / q]_{\zeta}(z) - [f; p / q]_{\zeta}(z)| \leq $$
$$ \leq \frac{2}{(\delta')^2}|A(f, \zeta)(z)B(g, \zeta)(z) - B(f, \zeta)(z)A(g, \zeta)(z)| $$

It follows that for $a > 0$ sufficiently small, the term:
$$ \frac{2}{(\delta')^2}|A(f, \zeta)(z)B(g, \zeta)(z) - B(f, \zeta)(z)A(g, \zeta)(z)| $$ 

can become arbitrary small. Hence $g \in B(p, q, j, s)$ and thus $B(p, q, j, s)$ is open in $X^{\infty}(\Omega)$.

\subsection{The case of \boldmath{$E(p, j, s)$.}} 

We fix the parametrs $(p, q) \in F$ and $j, s \geq 1$. Let $f \in E(p, j, s)$, $g \in X^{\infty}(\Omega)$ and $a > 0$ small enough. The number $a > 0$ will be determined later on. Suppose that the following holds:
$$ \sup_{z \in L_{n_0}} |f^{(m)}(z) - g^{(m)}(z)| < a $$
$$ \text{for every} \; m = 0, 1, 2, \cdots, \; max(s, p + q + 1) \;\;\; (7.3.1) $$

We will show that if $a > 0$ is small enough we obtain that $g \in E(p, j, s)$.

Our aim is to prove that for the right choice of $a > 0$, it holds:
$$ \sup_{\zeta \in L} \sup_{z \in K} |S_{p}(g, \zeta)(z) - S_{p}(f, \zeta)(z)| < $$
$$ < \frac{1}{s} - \sup_{\zeta \in L} \sup_{z \in K} |S_{p}(f, \zeta)(z) - f_{j}(z)| \; \;\; (7.3.2) $$ 

The last inequality is valid, provided that $a > 0$ is sufficiently small and by using relation $(7.3.1)$. The result follows from the triangle inequality.

\subsection{The case of \boldmath{$F(p, s, n)$.}} 

We fix the parametrs $(p, q) \in F$ and $s, n \geq 1$. Let $f \in F(p, s, n)$, $g \in X^{\infty}(\Omega)$ and $a > 0$ small enough. The number $a > 0$ will be determined later on. Suppose that the following holds:
$$ \sup_{z \in L_{n_0}} |f^{(m)}(z) - g^{(m)}(z)| < a $$
$$ \text{for every} \; m = 0, 1, 2, \cdots, \; max(s, p + q + 1) \;\;\; (7.4.1) $$

We will show that if $a > 0$ is small enough we obtain that $g \in F(p, s, n)$.

Our aim is to prove that for the right choice of $a > 0$, it holds:
$$ \sup_{\zeta \in L} \sup_{z \in K} |S_{p}^{(l)}(g, \zeta)(z) - S_{p}^{(l)}(f, \zeta)(z)| < $$
$$ < \frac{1}{s} - \sup_{\zeta \in L} \sup_{z \in K} |S_{p}^{(l)}(f, \zeta)(z) - f^{(l)}(z)| - \sup_{\zeta \in L} \sup_{z \in K} |f^{(l)}(z) - g^{(l)}(z)| $$
$$ \text{for every} \; l = 0, 1, 2, \cdots, s \; \;\; (7.4.2) $$ 

The last inequality is valid, provided that $a > 0$ is small enough and by using relation $(7.4.1)$ because the coefficients of the polynomials $S_{p}(f, \zeta)$ and $S_{p}(g, \zeta)$ are close enough one by one. It follows that $g \in F(p, s, n)$ and so $F(p, s, n)$ is a dense subset of $X^{\infty}(\Omega)$.

The proof of Claim \ref{claim 7.2} is complete.

\subsection{Density of \boldmath{$\mathcal{A}(j, s, n)$.}} 

In order to prove Claim \ref{claim 7.3} we fix the parameters and $j, s, n \geq 1$ and we want to prove that the set: 
$$ \mathcal{A}(j, s, n) = \bigcup_{(p, q) \in F} A(n, p, q, s) \cap B(p, q, j, s) \cap E(p, j, s) \cap F(p, s, n) $$ 

is a dense subset of $X^{\infty}(\Omega)$. 

Let $g \in X^{\infty}(\Omega)$, $\varepsilon > 0$ and $N \in \mathbb{N}$. We know that there exist an index $n_0 \in \mathbb{N}$ such that $L \cup L_n \subseteq L_{n_0}$. From the definiton of $X^{\infty}(\Omega)$, since the set $\{ \infty \} \cup (\mathbb{C} \setminus \overline{\Omega})$ is connected, it suffices to consider the function $g$ to be a polynomial.

We want to find a function $f \in \mathcal{A}(j, s, n)$ so that the following holds:
$$ \sup_{z \in L_{n_0}} |f^{(m)}(z) - g^{(m)}(z)| < \varepsilon \; \text{for every} \; m = 0, 1, 2, \cdots, N. $$

The sets $L_{n_0}$ and $K$ are disjoint compact subsets of $\mathbb{C}$ such that $\tilde{\mathbb{C}} \setminus L_{n_0}$ and $\tilde{\mathbb{C}} \setminus K$ are connected. In this case we know that there exist open and simply connected sets $G_1, G_2 \subseteq \mathbb{C}$ such that $L_{n_0} \subseteq G_1$, $K \subseteq G_2$ and $G_1 \cap G_2 = \emptyset$. Furthermore, we can also demand that $d(G_1, G_2) > 0$.

Consider now the function $w: G_1 \cup G_2 \to \mathbb{C}$ as defined below:
$$ w(z) = \begin{cases} f_j(z), \;\; \mbox{for every} \;\; z \in G_2 
\\ 
g(z),\;\; \mbox{for every} \;\; z \in G_1 \end{cases} $$

The function $w$ is well defined (because $G_1 \cap G_2 = \emptyset$) and also is holomorphic in $G_1 \cup G_2$.

We know from Runge's theorem that there exists a sequence of polynomials $\{P_n\}_{n \geq 1}$ that converges uniformly to $w$ on every compact subset of $G_1 \cup G_2$. Now, from Weierstrass's theorem (since $G_1 \cup G_2$ is open) we know that the previous convergence is also valid for every finite (non empty) set of derivatives. Thus, there exists a polynomial $P$ that is uniformly close to the funtion $f_j$ on $K$ while every polynomial $P^{(l)}$ is uniformly close to the function $g^{(l)}$ on $L_{n_0}$ and that holds for every $l = 0, 1, 2, \cdots, N$.

Let $(p, q) \in F$ such that $p > deg(P(z))$. We notice that for every $d \in \mathbb{C} \setminus \{ 0 \}$ it holds $deg(P(z) + dz^p) = p$, so, for every $q \geq 0$ it follows that $P(z) + dz^p \in D_{p, q}(\zeta)$ and also $[P(z) + dz^p; p / q]_{\zeta}(z) \equiv P(z) + dz^p$ for every $\zeta, z \in \mathbb{C}$, according to Proposition \ref{proposition 2.4}. If $d \to 0$, the polynomial $P(z) + dz^p$ converges uniformly to the polynomial $P(z)$ on every compact subset of $\mathbb{C}$.

We set $f(z) = P(z) + dz^p$ for $|d| > 0$ sufficiently small. It remains to show that $f \in \mathcal{A}(j, s, n)$. This is almost obvious; we have:

\begin{itemize}

\item[(1)]
$f \in A(n, p, q, s)$, since $[f; p / q]_{\zeta} \equiv f$ for ever $\zeta \in L$

\item[(2)]
$f \in B(p, q, j, s)$, since $[f; p / q]_{\zeta} \equiv f$ for ever $\zeta \in L$ and thus the quantity:
$$ \sup_{\zeta \in L} \sup_{z \in K} |[f; p / q]_{\zeta}(z) - f_{j}(z)| $$

can be arbitrary small (provided that $|d|$ is small enough)

\item[(3)]
$f \in E(p, j, s)$, since $f \equiv S_{p}(f, \zeta)$ for evey $\zeta \in L$ and so:
$$ \sup_{\zeta \in L} \sup_{z \in K} |S_{p}(f, \zeta)(z) - f_{j}(z)| = 0 $$

\item[(4)]
$f \in F(p, s, n)$, since $f \equiv S_{p}(f, \zeta)$ for evey $\zeta \in L$ and so:
$$ \sup_{\zeta \in L} \sup_{z \in L_n} |S_{p}^{(l)}(f, \zeta)(z) - f^{(l)}(z)| = 0 $$

\end{itemize}

The proof of Claim \ref{claim 7.3} is complete. The result follows from Baire's theorem.

\end{proof}

\begin{theorem} \label{theorem 7.4}

Let $F \subseteq \mathbb{N} \times \mathbb{N}$ be a non empty set that contains exactly a sequence $(p_n, q_n)_{n \geq 1}$ such that $p_n \to + \infty$. Also, let $\Omega \subseteq \mathbb{C}$ be a domain such that $\{ \infty \} \cup (\mathbb{C} \setminus \overline{\Omega})$ is connected and $\zeta \in \Omega$ be a fixed element. Then there exists a function $f \in X^{\infty}(\Omega)$ satisfying the following:

For every compact set $K \subseteq \mathbb{C} \setminus \overline{\Omega}$ and for every function $h \in A(K)$ there exists a subsequence $(p_{k_n}, q_{k_n})_{n \geq 1}$ of the sequence $(p_n, q_n)_{n \geq 1}$ such that:

\begin{itemize}

\item[(1)]
$f \in D_{p_{k_n}, q_{k_n}}(\zeta)$ for every $n \in \mathbb{N}$

\item[(2)]
For every $l \in \mathbb{N}$ it holds:
$$ \sup_{z \in J} |{[f; p_{k_n} / q_{k_n}]^{(l)}_{\zeta}}(z) - f^{(l)}(z)| \to 0 \; \text{as} \; n \to + \infty $$ 

for every compact set $J \subseteq \overline{\Omega}$.

\item[(3)]
$\sup_{z \in K} |[f; p_{k_n} / q_{k_n}]_{\zeta}(z) - h(z)| \to 0$ as $n \to + \infty$

\item[(4)]
$\sup_{z \in K} |S_{p_{k_n}}(f, \zeta)(z) - h(z)| \to 0$ as $n \to + \infty$

\item[(5)]
For every $l \in \mathbb{N}$ it holds:
$$ \sup_{z \in J} |S^{(l)}_{p_{k_n}}(f, \zeta)(z) - f^{(l)}(z)| \to 0 \; \text{as} \; n \to + \infty $$ 

for every compact $J \subseteq \overline{\Omega}$.

\end{itemize}

Moreover, the set of all functions satisfying $(1) - (5)$ is dense and $G_{\delta}$ in $X^{\infty}(\Omega)$. 

\end{theorem}

\begin{proof} It suffices to apply Theorem \ref{theorem 7.1} for $L = \{ \zeta \}$ and $K = K_n$ for $n \geq 1$ given by Lemma \ref{lemma 6.1}. In that way we find a $G_{\delta}$ - dense subset $\mathcal{A}_n$ of $X^{\infty}(\Omega)$. Then the set $\mathcal{A} = \cap_{n \geq 1} \mathcal{A}_n$ is also a dense and $G_{\delta}$ subset of $X^{\infty}(\Omega)$, according to Baire's theorem. But $\mathcal{A}$ is exactly the set of all functions sattisfying Theorem \ref{theorem 7.4}.

\end{proof}

\begin{theorem} \label{theorem 7.5}

Let $F \subseteq \mathbb{N} \times \mathbb{N}$ be a non empty set that contains exactly a sequence $(p_n, q_n)_{n \geq 1}$ such that $p_n \to + \infty$. Also, let $\Omega \subseteq \mathbb{C}$ be a domain such that $\{ \infty \} \cup (\mathbb{C} \setminus \overline{\Omega})$ is connected. Then there exists a function $f \in X^{\infty}(\Omega)$ satisfying the following:

For every compact set $K \subseteq \mathbb{C} \setminus \overline{\Omega}$ with connected complement and for every function $h \in A(K)$, there exists a subsequence $(p_{k_n}, q_{k_n})_{n \geq 1}$ of the sequence $(p_n, q_n)_{n \geq 1}$ such that for every compact set $L \subseteq \overline{\Omega}$ there exists an index $n \equiv n(L) \in \mathbb{N}$ so that the following hold:

\begin{itemize}

\item[(1)]
$f \in D_{p_{k_n}, q_{k_n}}(\zeta)$ for every $n \geq n(L)$ and for every $\zeta \in L$

\item[(2)]
For every $l \in \mathbb{N}$ it holds:
$$ \sup_{\zeta \in L} \sup_{z \in L} |[f; p_{k_n} / q_{k_n}]^{(l)}_{\zeta}(z) - f^{(l)}(z)| \to 0 \; \text{as} \; n \to + \infty $$

\item[(3)]
$\sup_{\zeta \in L} \sup_{z \in K} |[f; p_{k_n} / q_{k_n}]_{\zeta}(z) - h(z)|\to 0$ as $n \to + \infty$

\item[(4)]
$ \sup_{\zeta \in L} \sup_{z \in K} |S_{p_{k_n}}(f, \zeta)(z) - h(z)| \to 0$ as $n \to + \infty$

\item[(5)]
For every $l \in \mathbb{N}$ it holds:
$$ \sup_{\zeta \in L} \sup_{z \in L} |S^{(l)}_{p_{k_n}}(f, \zeta)(z) - f^{(l)}(z)| \to 0 \; \text{as} \; n \to + \infty $$

\end{itemize}

Moreover, the set of all functions satisfying $(1) - (5)$ is dense and $G_{\delta}$ in $X^{\infty}(\Omega)$. 

\end{theorem}

\begin{proof} We apply Theorem \ref{theorem 7.1} for $L = L_k = \overline{\Omega} \cap \overline{B}(0, k)$ and $K = K_n$ given by Lemma \ref{lemma 6.1} and we obtain a $G_{\delta}$ - dense subset of $X^{\infty}(\Omega)$; the set $\mathcal{A}_{k, n}$. Then the intersection $\mathcal{A} = \cap_{k, n \geq 1} \mathcal{A}_{k, n}$ is also a $G_{\delta}$ - dense set subset of $X^{\infty}(\Omega)$ according to Baire's theorem. But $\mathcal{A}$ is exactly the set of all functions satisfying Theorem \ref{theorem 7.5}.

\end{proof}

\section{Splitting the boundary} 

In this section we combine and strengthen the results of \cite{enteka} and \cite{Kar.Nest.}. We consider an open set $\Omega \subseteq \mathbb{C}$ and $\{ L_n \}_{n \geq 1}$ a sequence of compact subsets of $\overline{\Omega}$ satisfying the following properties:

\begin{itemize}

\item[(1)]
$\overline{L_n \cap \Omega} = L_n$ for every $n \geq 1$.

\item[(2)]
Each connected component of $\{ \infty \} \cup \mathbb{C} \setminus L_n$ contains at least a connected component of $\{ \infty \} \cup \mathbb{C} \setminus \overline{\Omega}$.

\item[(3)]
Every compact subset of $\Omega$ is contained in one of the sets $\{ L_n \}_{n \geq 1}$.

\end{itemize}

Let $T^{\infty}(\Omega) \equiv T^{\infty}(\Omega, \{ L_n \}_{n \geq 1})$ be the space of all functions $f \in H(\Omega)$ such that for every defivative $f^{(l)}$ ($l \geq 0$) of $f$ and for every $L_n$ ($n \geq 1$), the restriction $f^{(l)}_{| L_n \cap \Omega}$ is uniformly continuous and therefore it extends continuously on $\overline{L_n \cap \Omega} = L_n$.

We endow this space with the seminorms $\sup_{z \in L_n} |f^{(l)}(z)|$ for $l \geq 0$ and for $n \geq 1$. In that way, $T^{\infty}(\Omega)$ becomes a Fr\'{e}chet space, containinig all rational functions with poles off the set $\cup_{n \geq 1} L_n$.

Consider now $Y^{\infty}(\Omega)$ to be the closure in $T^{\infty}(\Omega)$ all rational functions with poles off $\cup_{n \geq 1} L_n$. Since $Y^{\infty}(\Omega)$ is a closed subset of a complete metric space, it is also a complete metric space itself.

The reader is prompted to verify the following:

\begin{itemize}

\item[(1)]
If $\{ \infty \} \cup \mathbb{C} \setminus \overline{\Omega}$ is connected, then the set of polynomials (as elements of $T^{\infty}(\Omega)$) is a dense subset of $Y^{\infty}(\Omega)$.

\item[(2)]
For every $f \in Y^{\infty}(\Omega)$ and for every basic open set $\mathcal{B}$ containing $f$ there exists another basic open set of the form:
$$ \{ g \in Y^{\infty}(\Omega) : \sup_{z \in L_M} |f^{(l)}(z) - g^{(l)}(z)| < \varepsilon \; \text{for every} \; l = 0, 1, \cdots, N \} $$

for an appropriate choice of $M, N \geq 1$ and $\varepsilon > 0$, that is contained in $\mathcal{B}$.

\end{itemize}

We present now the main theorem of this section.

\begin{theorem} \label{theorem 8.1}

Let $F \subseteq \mathbb{N} \times \mathbb{N}$ be a non empty set that contains exactly a sequence $(p_n, q_n)_{n \geq 1}$ such that $p_n \to + \infty$. Also, let $\Omega \subseteq \mathbb{C}$ be an open set such that $\{ \infty \} \cup \mathbb{C} \setminus \overline{\Omega}$ is connected and $K \subseteq \mathbb{C}$ be a compact set with connected complement such that $K \cap L_n = \emptyset$ for every $n \in \mathbb{N}$. In addition, let $m \in \mathbb{N}$ be a fixed natural number. 

Then there exists a function $f \in Y^{\infty}(\Omega)$ such that for every polynomial $P$ there exists a subsequence $(p_{k_n}, q_{k_n})_{n \geq 1}$ of $(p_n, q_n)_{n \geq 1}$ satisfying the following properties:

\begin{itemize}

\item[(1)]
$f \in D_{p_{k_n}, q_{k_n}}(\zeta)$ for every $\zeta \in L_m$

\item[(2)]
For every $l \in \mathbb{N}$ it holds:
$$ \sup_{\zeta \in L_m} \sup_{z \in L_m} |[f; p_{k_n} / q_{k_n}]^{(l)}_{\zeta}(z) - f^{(l)}(z)| \to 0 \; \text{as} \; n \to + \infty $$

\item[(3)]
For every $l \in \mathbb{N}$ it holds:
$$ \sup_{\zeta \in L_m} \sup_{z \in K} |[f; p_{k_n} / q_{k_n}]_{\zeta}^{(l)}(z) - P^{(l)}(z)| \to 0 \; \text{as} \; n \to + \infty $$

\item[(4)]
For every $l \in \mathbb{N}$ it holds:
$$ \sup_{\zeta \in L_m} \sup_{z \in L_m} |S^{(l)}_{p_{k_n}}(f, \zeta)(z) - f^{(l)}(z)| \to 0 \; \text{as} \; n \to + \infty $$

\item[(5)]
For every $l = 0, 1, 2, \cdots$ it holds:
$$ \sup_{\zeta \in L_m} \sup_{z \in K} |S^{(l)}_{p_{k_n}}(f, \zeta)(z) - P^{(l)}(z)| \to 0 \; \text{as} \; n \to + \infty $$
\end{itemize}

Moreover, the set of all functions satisfying $(1) - (5)$ is dense and $G_{\delta}$ in $Y^{\infty}(\Omega)$. 

\end{theorem}

\begin{proof} Let $\{ f_j \}_{j \geq 1}$ be an enumeration of polynomials with coefficients in $\mathbb{Q} + i \mathbb{Q}$. 

Let also $\mathcal{U}$ be the set of all functions in $Y^{\infty}(\Omega)$ satisfying $(1) - (5)$. Our aim is to prove that $\mathcal{U}$ is a dense and $G_{\delta}$ subset of $Y^{\infty}(\Omega)$. 

Now, for every $(p, q) \in F$ and for every $j, s \geq 1$ we consider the following sets: 
$$ A(p, q, s) = \{ f \in Y^{\infty}(\Omega) : f \in D_{p, q}(\zeta) \; \text{for every} \; \zeta \in L_m \; \text{and} $$ 
$$ \sup_{\zeta \in L_m} \sup_{z \in L_m} |[f; p / q]_{\zeta}^{(l)}(z) - f^{(l)}(z)| < \frac{1}{s} \; \text{for every} \; l = 0, 1, 2, \cdots, s \} $$
$$ B(p, q, j, s) = \{ f \in Y^{\infty}(\Omega) : f \in D_{p, q}(\zeta) \; \text{for every} \; \zeta \in L_m \; \text{and} $$
$$ \sup_{\zeta \in L_m} \sup_{z \in K} |[f; p / q]_{\zeta}^{(l)}(z) - f_{j}^{(l)}(z)| < \frac{1}{s} \; \text{for every} \; l = 0, 1, 2, \cdots, s \} $$
$$ E(p, s) = \{ f \in Y^{\infty}(\Omega) : \sup_{\zeta \in L_m} \sup_{z \in L_m} |S_{p}^{(l)}(f, \zeta)(z) - f^{(l)}(z)| < \frac{1}{s} $$
$$ \text{for every} \; l = 0, 1, 2, \cdots, s \} $$
$$ F(p, j, s) = \{ f \in Y^{\infty}(\Omega) : \sup_{\zeta \in L_m} \sup_{z \in K} |S_{p}^{(l)}(f, \zeta)(z) - f_{j}^{(l)}(z)| < \frac{1}{s} $$
$$ \text{for every} \; l = 0, 1, 2, \cdots, s \}$$

One can verify that the following holds:
$$ \mathcal{U} = \bigcap_{j, s \geq 1} \bigcup_{(p, q) \in F} A(p, q, s) \cap B(p, q, j, s) \cap E(p, s) \cap F(p, j, s) $$

So, according to Baire's theorem it is enough to prove the following:

\begin{claim} \label{claim 8.2}

The sets $A(p, q, s), B(p, q, j, s), E(p, s)$ and $F(p, j, s)$ are open subsets of $Y^{\infty}(\Omega)$ for every parameter $(p, q) \in F$ and for every $j, s \geq 1$.

\end{claim}

\begin{claim} \label{claim 8.3}

The set $\mathcal{A}(j, s) = \bigcup_{(p, q) \in F} A(p, q, s) \cap B(p, q, j, s) \cap E(p, s) \cap F(p, j, s)$ is a dense subset of $Y^{\infty}(\Omega)$ for every parameter $j, s \geq 1$.

\end{claim}

\subsection{The case of \boldmath{$A(p, q, s)$.}} 

We fix the parameters $(p, q) \in F$ and $s \geq 1$. Let $f \in A(p, q, s)$, $g \in Y^{\infty}(\Omega)$ and $a > 0$ be small. The number $a > 0$ will be determined later on. Suppose that the following holds:
$$ \sup_{z \in L_m} |f^{(l)}(z) - g^{(l)}(z)| < a $$
$$ \text{for every} \; l = 0, 1, 2, \cdots, p + q + s \;\;\; (8.1.1) $$

We will show that if $a > 0$ is small enough we obtain that $g \in A(p, q, s)$.

Since $f \in A(p, q, s)$, the Hankel determinant $D_{p, q}(f, \zeta)$ for $f$ is not equal to zero and depends continuously on $\zeta \in L_m$; thus there exists a $\delta > 0$ such that $|D_{p, q}(f, \zeta)| > \delta$ for every $\zeta \in L_m$. From relation $(8.1.1)$ if $a > 0$ is  sufficiently small then the Hankel determinant $D_{p, q}(g, \zeta)$ for $g$ is greater in absolute value than $\frac{\delta}{2} > 0$ and this holds for every $\zeta \in L_m$. In other words, $g \in D_{p, q}(\zeta)$ for every $\zeta \in L_m$.

On the other hand, $f(z) \in \mathbb{C}$ for every $z \in L_m$; so $[f; p / q]_{\zeta}(z) \in \mathbb{C}$ for every $\zeta \in L_m$ and for every $z \in L_m$. It follows that $B(f, \zeta)(z) \neq 0$ for every $\zeta \in L_m$ and for every $z \in L_m$, where the polynomial $B(f, \zeta)(z)$ is defined according to the Jacobi formula. Thus, we have that there exists a $0 < \delta'' < 1$ such that $|B(f, \zeta)(z)| > \delta''$ for every $(\zeta, z) \in L_m \times L_m$ (since the function $B(f, \zeta)(z)$ is continuous on $L_m \times L_m$). By continuity and for $a > 0$ sufficiently small, one can achieve the following: 
$$ |B(g, \zeta)(z)| > \frac{\delta''}{2} \; \text{for every} \; (\zeta, z) \in L_m \times L_m $$
 
For every $l = 0, 1, 2, \cdots, p + q + s$ and from the triangle enequality, we obtain: 
$$ \sup_{\zeta \in L_m} \sup_{z \in L_m} |[g; p / q]_{\zeta}^{(l)}(z) - g^{(l)}(z)| \leq $$
$$  \sup_{\zeta \in L_m} \sup_{z \in L_m} |[g; p / q]_{\zeta}^{(l)}(z) - [f; p / q]_{\zeta}^{(l)}(z)| $$
$$ + \sup_{\zeta \in L_m} \sup_{z \in L_m} |[f; p / q]_{\zeta}^{(l)}(z) - f^{(l)}(z)| $$
$$ + \sup_{z \in L_m} |f^{(l)}(z) - g^{(l)}(z)| \;\;\; (8.1.2) $$

The term $\sup_{z \in L_m} |f^{(l)}(z) - g^{(l)}(z)|$ can become as small as we want for every $l = 0, 1, 2, \cdots, p + q + s$, provided that $a > 0$ is sufficiently small.

The term $\sup_{\zeta \in L_m} \sup_{z \in L_m} |[f; p / q]_{\zeta}^{(l)}(z) - f^{(l)}(z)|$ is strictly less than $\frac{1}{s}$ for every $l = 0, 1, 2, \cdots, p + q + s)$ since $f \in A(p, q, s)$. It remains to control the term:
$$ \sup_{\zeta \in L_m} \sup_{z \in L_m} |[g; p / q]_{\zeta}^{(l)}(z) - [f; p / q]_{\zeta}^{(l)}(z)| $$

and that for every $l = 0, 1, 2, \cdots, p + q + s$.

The denominators of $[f; p / q]_{\zeta}^{(l)}$ and $[g; p / q]_{\zeta}^{(l)}$ are uniformly bounded far from $0$ for every $l = 0, 1, 2, \cdots, p + q + s$. 

So, the quantity $\sup_{\zeta \in L_m} \sup_{z \in L_m} |[g; p / q]_{\zeta}^{(l)}(z) - [f; p / q]_{\zeta}^{(l)}(z)|$ can become arbitrary small, since the coefficients of the polynomials $A(f, \zeta)$ and $A(g, \zeta)$ are uniformly close enough one by one. The same happens for the polynomials $B(f, \zeta)$ and $B(g, \zeta)$. This allows us to control every finite set of derivatives. It follows that $g \in A(p, q, s)$ and so $A(p, q, s)$ is an open subset of $Y^{\infty}(\Omega)$.

\subsection{The case of \boldmath{$B(p, q, j, s)$.}} 

We fix the parameters $(p, q) \in F)$ and $j, s \geq 1$. Let $f \in B(p, q, j, s)$, $g \in Y^{\infty}(\Omega)$ and $a > 0$ be small. The number $a > 0$ will be determined later on. Suppose that the following holds:
$$ \sup_{z \in L_m} |f^{(l)}(z) - g^{(l)}(z)| < a $$
$$ \text{for every} \; l = 0, 1, 2, \cdots, p + q + s \;\;\; (8.2.1) $$

We will show that if $a > 0$ is small enough we obtain that $g \in B(p, q, j, s)$.

In order to prove that $g \in D_{p, q}(\zeta)$ for every $\zeta \in L_m$ we follow exactly the same steps as we did in subsection $8.1$.

Since $f \in B(p, q, j, s)$ we have that it holds $\sup_{\zeta \in L_m} \sup_{z \in K} |[f; p / q]_{\zeta}^{(l)}(z) - f_{j}^{(l)}(z)| < \frac{1}{s}$ for every $\zeta \in L_m$, for every $z \in K$ and for every $l = 0, 1, 2, \cdots, p + q + s$; thus $[f; p / q]_{\zeta}^{(l)}(z) \in \mathbb{C}$ for every $\zeta \in L_m$, for every $z \in K$ and for every $l = 0, 1, 2, \cdots, p + q + s$. It follows that there exists a $\delta' > 0$ such that $|B(f, \zeta)(z)| > \delta'$ for every $\zeta \in L_m$ and for every $z \in K$, since $L_m \times K$ is a compact set, where $B(f, \zeta)$ is the denominator of the Jacobi formula for $f$. From relation $(8.2.1)$ and the Jacobi formula it follows that for $a > 0$ sufficiently small, it holds:
$$ |B(g, \zeta)(z)| > \frac{\delta'}{2} \; \text{for every} \; (\zeta, z) \in L_m \times K $$

In order to complete the proof it suffices to prove the following:
$$ \sup_{\zeta \in L_m} \sup_{z \in K} |[g; p / q]_{\zeta}^{(l)}(z) - [f; p / q]_{\zeta}^{(l)}(z)| < $$
$$ < \frac{1}{s} - \sup_{\zeta \in L_m} \sup_{z \in K} |[f; p / q]_{\zeta}^{(l)}(z) - f_{j}^{(l)}(z)| \;\;\; (8.2.2) $$

for every $l = 0, 1, 2, \cdots, p + q + s$ and then the result yields from the triangle inequality. This is valid, since the term:
$$ \sup_{\zeta \in L_m} \sup_{z \in K} |[g; p / q]_{\zeta}^{(l)}(z) - [f; p / q]_{\zeta}^{(l)}(z)| $$ 

can became arbitrary small. It follows that $g \in B(p, q, j, s)$ and so $B(p, q, j, s)$ is an open subset of $Y^{\infty}(\Omega)$.

\subsection{The case of \boldmath{$E(p, s)$.}} 

We fix the parameters $(p, q) \in F$ and $j, s \geq 1$. Let $f \in E(p, s)$, $g \in Y^{\infty}(\Omega)$ and $a > 0$ small enough. The number $a > 0$ will be determined later on. Suppose that the following holds:
$$ \sup_{z \in L_m} |f^{(l)}(z) - g^{(l)}(z)| < a $$
$$ \text{for every} \; l = 0, 1, 2, \cdots, p + q + s \;\;\; (8.3.1) $$

We will show that if $a > 0$ is small enough we obtain that $g \in E(p, s)$.

Our aim is to prove that for the right choise of $a > 0$, it holds:
$$ \sup_{\zeta \in L_m} \sup_{z \in L_m} |S_{p}^{(l)}(g, \zeta)(z) - S_{p}^{(l)}(f, \zeta)(z)| < $$
$$ < \frac{1}{s} - \sup_{\zeta \in L_m} \sup_{z \in L_m} |S_{p}^{(l)}(f, \zeta)(z) - f^{(l)}(z)| $$
$$ \text{for every} \; l = 0, 1, 2, \cdots, p + q + s \;\;\; (8.3.2) $$ 

The last inequality is valid, provided that $a > 0$ is sufficiently small and by using relation $(8.3.1)$. The result yields from the triangle inequality.

\subsection{The case of \boldmath{$F(p, j, s)$.}} 

We fix the parameters $(p, q) \in F)$ and $s \geq 1$. Let $f \in F(p, s)$, $g \in Y^{\infty}(\Omega)$ and $a > 0$ small enough. The number $a > 0$ will be determined later on. Suppose that the following holds:
$$ \sup_{z \in L_m} |f^{(l)}(z) - g^{(l)}(z)| < a $$
$$ \text{for every} \; l = 0, 1, 2, \cdots, p + q + s \;\;\; (8.4.1) $$

We will show that if $a > 0$ is small enough we obtain that $g \in F(p, s)$.

Our aim is to prove that for the right choise of $a > 0$, it holds:
$$ \sup_{\zeta \in L_m} \sup_{z \in K} |S_{p}^{(l)}(g, \zeta)(z) - S_{p}^{(l)}(f, \zeta)(z)| < $$
$$ < \frac{1}{s} - \sup_{\zeta \in L_m} \sup_{z \in K} |S_{p}^{(l)}(f, \zeta)(z) - f_j^{(l)}(z)| $$
$$ \text{for every} \; l = 0, 1, 2, \cdots, p + q + s \; \;\; (8.4.2) $$ 

The last inequality is valid, provided that $a > 0$ is small enough and by using relation $(8.4.1)$ because the coefficients of the polynomials $S_{p}(f, \zeta)$ and $S_{p}(g, \zeta)$ are uniformly close enough one by one. The result follows from the triangle inequality.

The proof of Claim \ref{claim 8.2} is complete.

\subsection{Density of \boldmath{$\mathcal{A}(j, s)$.}} 

In order to prove Claim \ref{claim 8.3}, we fix the parameters $j, s \geq 1$ and we want to prove that the set: 
$$ \mathcal{A}(j, s) = \bigcup_{(p, q) \in F} A(p, q, s) \cap B(p, q, j, s) \cap E(p, s) \cap F(p, j, s) $$ 

is a dense subset of $Y^{\infty}(\Omega)$.

Let $g \in Y^{\infty}(\Omega)$, $\varepsilon > 0$ and $M, N \in \mathbb{N}$. We want to find a function $f \in \mathcal{A}(j, s)$ so that the following holds:
$$ \sup_{z \in L_M} |f^{(l)}(z) - g^{(l)}(z)| < \varepsilon \; \text{for every} \; m = 0, 1, 2, \cdots, N $$

Since the sequence $\{ L_m \}_{m \geq 1}$ is increasing, there is no problem to assume that $m < M$ (and thus $L_m \subseteq L_M$). Moreover, we assume that $g$ is a polynomial, since $\{ \infty \} \cup \mathbb{C} \setminus \overline{\Omega}$ is connected (and so, as we have already mentioned the polynomials are a dense subset of $Y^{\infty}(\Omega)$).

The sets $L_M$ and $K$ are disjoint compact subsets of $\mathbb{C}$. Since every connected component of $\{ \infty \} \cup \mathbb{C} \setminus L_N$ contains a connected component of $\{ \infty \} \cup \mathbb{C} \setminus \overline{\Omega}$ and according to our hypothesis, the set $\{ \infty \} \cup \mathbb{C} \setminus \overline{\Omega}$ is connected, we know that $\{ \infty \} \cup \mathbb{C} \setminus L_N$ is also a connected set. Since $K$ has connected complement and $K \cap L_N = \emptyset$, we know that there exist open and simply connected sets $G_1, G_2 \subseteq \mathbb{C}$ so that $L_M \subseteq G_1$, $K \subseteq G_2$ and $G_1 \cap G_2 = \emptyset$. 

Consider now the function $w: G_1 \cup G_2 \to \mathbb{C}$ as defined below:
$$ w(z) = \begin{cases} f_j(z), \;\; \mbox{for every} \;\; z \in G_2 
\\ 
g(z),\;\; \mbox{for every} \;\; z \in G_1 \end{cases} $$

The function $w$ is well - defined (because $G_1 \cap G_2 = \emptyset$) and also is holomorphic in $G_1 \cup G_2$.

We know from Runge's theorem that there exists a sequence of polynomials $\{P_n\}_{n \geq 1}$ that converges uniformly to $w$ on every compact subset of $G_1 \cup G_2$. Now, from Weierstrass's theorem (since $G_1 \cup G_2$ is open) we know that the previous convergence is also valid for every derivative. So, there exists a polynomial $P$ that is uniformly close to the funtion $f_j$ on $K$ while every polynomial $P^{(l)}$ is uniformly close to the function $g^{(l)}$ on $L_M$ and that holds for every $l = 0, 1, 2, \cdots, N$.

Let $(p, q) \in F$ such that $p > deg(P(z))$. Since $p_n \to + \infty$, one can suppose that $(p, q) = (p_{k_n}, q_{k_n})$ for an appropriate index ${k_n} \in \mathbb{N}$. We notice that for every $d \in \mathbb{C} \setminus \{ 0\}$ it holds $deg(P(z) + dz^p) = p$, thus, for every $q \geq 0$ it holds $P(z) + dz^p \in D_{p, q}(\zeta)$ and also $[P(z) + dz^p; p / q]_{\zeta}(z) \equiv P(z) + dz^p$ for every $\zeta \in \mathbb{C}$ according to Proposition \ref{proposition 2.4}. If $d \to 0$, the polynomial $P(z) + dz^p$ converges uniformly to the polynomial $P(z)$ on each compact subset of $\mathbb{C}$.

We set $f(z) = P(z) + dz^p$ for $|d| > 0$ sufficiently small. It remains to show that $f \in \mathcal{A}(j, s)$. This is almost obvious; we have:

\begin{itemize}

\item[(1)]
$f \in A(p, q, s)$, since $[f; p / q]_{\zeta} \equiv f$ for ever $\zeta \in L_m$ and so the quantity:
$$ \sup_{\zeta \in L_m} \sup_{z \in L_m} |[f; p / q]_{\zeta}^{(l)}(z) - f^{(l)}(z)| = 0 $$

\item[(2)]
$f \in B(p, q, j, s)$, since $[f; p / q]_{\zeta} \equiv f$ for every $\zeta \in L_m$ and so the quantity:
$$ \sup_{\zeta \in L_m} \sup_{z \in K} |[f; p / q]_{\zeta}^{(l)}(z) - f_{j}^{(l)}(z)| $$

can be arbitrary small (provided that $|d|$ is small enough), since it holds $[f; p / q]_{\zeta} \equiv f \equiv S_p(f)$.

\item[(3)]
$f \in E(p, s)$, since $f \equiv S_{p}(f, \zeta) \equiv [f; p / q]_{\zeta}$ for every $\zeta \in L_m$ and so the quantity:
$$ \sup_{\zeta \in L_m} \sup_{z \in L_m} |S_{p}^{(l)}(f, \zeta)(z) - f_{j}^{(l)}(z)| $$

is exactly the same as in $(2)$.

\item[(4)]
$f \in F(p, j, s)$, since $f \equiv S_{p}(f, \zeta)$ for evey $\zeta \in L_m$ and so the quantity:
$$ \sup_{\zeta \in L_m} \sup_{z \in K} |S_{p}^{(l)}(f, \zeta)(z) - f^{(l)}(z)| = 0$$
\end{itemize}

The proof of Claim \ref{claim 8.3} is also complete. The result follows from Baire's theorem.

\end{proof}

We fix a sequence of compact sets $\{ K_m \}_{m \geq 1}$ with connected complements such that $K_m \cap L_n = \emptyset$ for every $n, m \geq 1$.

\begin{theorem} \label{theorem 8.4}

Let $F \subseteq \mathbb{N} \times \mathbb{N}$ be a non empty set that contains exactly a sequence $(p_n, q_n)_{n \geq 1}$ such that $p_n \to + \infty$. Also, let $\Omega \subseteq \mathbb{C}$ be an open set such that $\{ \infty \} \cup \mathbb{C} \setminus \overline{\Omega}$ is connected.

Then there exists a function $f \in Y^{\infty}(\Omega)$ such that for every compact set $K_m$ and for every polynomial $P$ there exists a subsequence $(p_{k_n}, q_{k_n})_{n \geq 1}$ of the sequence $(p_n, q_n)_{n \geq 1}$ such that for every compact set $L \subseteq \overline{\Omega}$ with $L \subseteq L_m$ for an index $m \geq 1$, there exists an index $n \equiv n(L) \in \mathbb{N}$ so that the following hold:

\begin{itemize}

\item[(1)]
$f \in D_{p_{k_n}, q_{k_n}}(\zeta)$ for every $\zeta \in L$ and for every $n \geq n(L)$ 

\item[(2)]
For every $l \in \mathbb{N}$ it holds:
$$ \sup_{\zeta \in L} \sup_{z \in L} |[f; p_{k_n} / q_{k_n}]^{(l)}_{\zeta}(z) - f^{(l)}(z)| \to 0 \; \text{as} \; n \to + \infty $$ 

\item[(3)]
For every $l \in \mathbb{N}$ it holds:
$$ \sup_{\zeta \in L} \sup_{z \in K_m} |[f; p_{k_n} / q_{k_n}]_{\zeta}^{(l)}(z) - P^{(l)}(z)| \to 0 \; \text{as} \; n \to + \infty $$

\item[(4)]
For every $l \in \mathbb{N}$ it holds:
$$ \sup_{\zeta \in L} \sup_{z \in L} |S^{(l)}_{p_{k_n}}(f, \zeta)(z) - f^{(l)}(z)| \to 0 \; \text{as} \; n \to + \infty $$

\item[(5)]
For every $l = 0, 1, 2, \cdots$ it holds:
$$ \sup_{\zeta \in L} \sup_{z \in K_m} |S^{(l)}_{p_{k_n}}(f, \zeta)(z) - P^{(l)}(z)| \to 0 \; \text{as} \; n \to + \infty $$ 

\end{itemize}

Moreover, the set of all functions satisfying $(1) - (5)$ is dense and $G_{\delta}$ in $Y^{\infty}(\Omega)$. 

\end{theorem}

\begin{proof} Let $L \subseteq \Omega$ be a compact set. There exists an index $n \in \mathbb{N}$ such that $L \subseteq L_n$. Now for every $m \geq 1$ we apply Theorem \ref{theorem 8.1} for $K = K_m$ and according to Baire's theorem we obtain a dense and $G_{\delta}$ set of $Y^{\infty}(\Omega)$; the set $\mathcal{A}_{n, m}$. If $\mathcal{A}$ is the set of all functions satisfying Theorem \ref{theorem 8.4} it follows from a diagonal argument that $\mathcal{A} = \cap_{n, m \geq 1} \mathcal{A}_{n, m}$.

\end{proof}

\section*{Acknowledgements}

I would like to thank my PhD supervisor professor Vassili Nestoridis for his precious help towards completing this draft.

\bigskip
\bigskip
\bigskip

\noindent
University of Athens
\\
Department of Mathematics
\\
157 84 Panepistimioupolis
\\
Athens
\\
Greece

\bigskip

\noindent
email address:
\\
K. Makridis (kmak167@gmail.com)

\end{document}